\documentclass{amsart}
\usepackage{amsfonts}
\usepackage{amssymb, amsmath, amsthm}
\usepackage{amscd}
\usepackage{verbatim}
\usepackage{enumerate}
\usepackage{color}
\usepackage{paralist}
\usepackage[colorlinks=true]{hyperref}

\allowdisplaybreaks
\newtheorem{theorem}{Theorem}[section]

\newtheorem{corollary}[theorem]{Corollary}

\newtheorem{definition}[theorem]{Definition}

\newtheorem{lemma}[theorem]{Lemma}

\newtheorem{proposition}[theorem]{Proposition}

\theoremstyle{remark}
\newtheorem{remark}[theorem]{Remark}

\numberwithin{equation}{section}

\begin{document}
\title{Longtime behavior of nonlocal Cahn-Hilliard equations}
\author{Ciprian G. Gal}
\address{Department of Mathematics, Florida International University, Miami,
FL 33199, USA}
\email{cgal@fiu.edu}
\author{Maurizio Grasselli}
\address{Dipartimento di Matematica ``F.~Brioschi'', Politecnico di Milano,
20133 Milano, Italy}
\email{maurizio.grasselli@polimi.it}
\maketitle

\begin{abstract}
Here we consider the nonlocal Cahn-Hilliard equation with constant mobility
in a bounded domain. We prove that the associated dynamical system has an
exponential attractor, provided that the potential is regular. In order to
do that a crucial step is showing the eventual boundedness of the order
parameter uniformly with respect to the initial datum. This is obtained
through an Alikakos-Moser type argument. We establish a similar result for
the viscous nonlocal Cahn-Hilliard equation with singular (e.g.,
logarithmic) potential. In this case the validity of the so-called
separation property is crucial. We also discuss the convergence of a
solution to a single stationary state. The separation property in the
nonviscous case is known to hold when the mobility degenerates at the pure
phases in a proper way and the potential is of logarithmic type. Thus, the
existence of an exponential attractor can be proven in this case as well.
\end{abstract}

\tableofcontents

\section{Introduction}

\noindent The Cahn-Hilliard equation was proposed in \cite{CH} as a model
for (isothermal) phase separation phenomena in binary alloys. Since then it
was analyzed by many authors and used in several different contexts (see,
e.g., \cite{CMZ,N-C2} and references therein). The basic form of such an
equation is the following
\begin{equation}
\partial _{t}\varphi =\nabla \cdot \lbrack \kappa (\varphi )\nabla \mu ],
\label{ch}
\end{equation}%
where $\varphi $ is the relative difference of the two phases (or the
concentration of one phase), and $\mu $ is the so-called chemical potential
given by
\begin{equation}
\mu =-\epsilon \Delta \varphi +\frac{1}{\epsilon }F^{\prime }\left( \varphi
\right) .  \label{chemloc}
\end{equation}%
in $\Omega \times \left( 0,\infty \right) ,$ where $\Omega \subset \mathbb{R}%
^{d}$, $d=2,3$, is a bounded domain with \emph{Lipschitz} boundary $\Gamma
=\partial \Omega $. Here $\kappa $ is the mobility coefficient, $\epsilon >0$
is a given (small) parameter related to the thickness of the interface
separating the two phases, and $F$ is the (density) of potential energy. A
physically relevant choice for $F$ is the following
\begin{equation}
F(r)=(1+r)\log (1+r)+(1-r)\log (1-r)-\lambda r^{2},\quad \lambda \geq 1,
\label{logpot}
\end{equation}%
which is often approximated by a polynomial double well-potential, typically
\begin{equation}
F(r)=(r^{2}-1)^{2}.  \label{dwpot}
\end{equation}%
In the literature, it is common to distinguish between singular potentials,
which are defined on finite intervals like \eqref{logpot}, and regular ones
as \eqref{dwpot}, defined on $\mathbb{R}$ .

We recall that equations \eqref{ch}-\eqref{chemloc} have been deduced
phenomenologically, i.e., as the (conserved) gradient flow associated with
the Fr\'{e}chet derivative of the free energy functional
\begin{equation*}
\mathcal{L}(\varphi) = \int_\Omega \left(\frac{\epsilon}{2}\vert \nabla
\varphi\vert^2 + \frac{1}{\epsilon}F(\varphi)\right)dx.
\end{equation*}

In \cite{GL1,GL2}, starting from a microscopic model, the authors rigorously
derived a macroscopic equation for phase segregation phenomena. This is a
nonlocal version of the Cahn-Hilliard equation, namely, the chemical
potential is given by
\begin{equation}
\mu =a_{\epsilon }(x)\varphi -\epsilon J_{\epsilon }\ast \varphi +\frac{1}{%
\epsilon }F^{\prime }\left( \varphi \right) ,  \label{chemnl}
\end{equation}%
where $J$ is a (sufficiently smooth) interaction kernel such that $%
J(x)=J(-x) $ and
\begin{equation*}
(J_{\epsilon }\ast \varphi )(x):=\int_{\Omega }J_{\epsilon }(x-y)\varphi
(y)dy,\quad a_{\epsilon }(x):=\epsilon \int_{\Omega }J_{\epsilon }(x-y)dy,
\end{equation*}%
where $J_{\epsilon }(x)=\epsilon ^{-d}J(\epsilon ^{-1}x)$. By using formal
asymptotic analysis, the authors also showed that the interface evolution
problems associated with such equation as $\epsilon $ goes to $0$ are
exactly the ones associated with the standard Cahn-Hilliard equation (i.e.,
Stefan-like and Mullins-Sekerka problems). In addition, also the nonlocal
version can be viewed as the conserved gradient flow associated with the
first variation of the free energy functional
\begin{equation*}
\mathcal{N}(\varphi )=\int_{\Omega \times \Omega }\frac{\epsilon }{4}%
J_{\epsilon }(x-y)|\varphi (x)-\varphi (y)|^{2}dxdy+\int_{\Omega }\frac{1}{%
\epsilon }F(\varphi )dx.
\end{equation*}%
As a consequence, we can observe (formally) that the nonlocal interaction
term can be locally approximated by the square gradient, provided that $J$
is sufficiently concentrated around $0$. That is, the functional $\mathcal{L}
$ can be viewed as a local approximation of $\mathcal{N}$. This was already
noted by Van der Waals (see \cite{Ro}). Thus the nonlocal Cahn-Hilliard
equation seems well justified and more general than the classical one,
though the related literature is far less abundant. In particular, most of
the theoretical results are devoted to well-posedness, but very few are
concerned with the longtime behavior of solutions. The main reason is
related to the eventual boundedness and regularization of the order
parameter which are needed to prove the precompactness of trajectories in
some convenient topology. Well-posedness and regularity issues were firstly
analyzed in \cite{GL2} on a three-dimensional torus with degenerate mobility
and logarithmic potential. A similar equation endowed with no flux boundary
condition was studied in \cite{GZ} (cf. also \cite{CKRS,Ga,GG} and, for
viscous versions, \cite{HFS1,HFS2}). For this case, the convergence to a
single stationary state of a given trajectory was proven in \cite{LP}
through a suitable {\L }ojasiewicz-Simon inequality. This fact required to
show preliminarily that a solution stays eventually strictly away from the
pure phases: the so-called separation property.

For the constant mobility case and regular potentials, some existence,
uniqueness and regularity results were obtained in \cite{BH} (see also \cite%
{BH2, Ha}). In that paper the existence of bounded absorbing sets was also
established. Nevertheless, no results were known about the existence of more
interesting invariant objects like, e.g., global attractors (cf. \cite{MZrev}
and its references). Only recently, the existence of a (connected) global
attractor has been proven in \cite{FG2} for constant mobility and regular
potentials (see \cite{FG1} for singular ones). This has been done by
exploiting the energy identity as a by-product of a result related to a more
complicated model for phase separation in binary fluids. A natural question
now arises: does the global attractor have \emph{finite} (fractal)
dimension? Here we give a positive answer and we actually prove more,
namely, the existence of an exponential attractor (see again \cite{MZrev}
for details). More precisely, taking for simplicity $\kappa =\epsilon =1$,
we consider the following nonlocal Cahn-Hilliard equation
\begin{align}
& \partial _{t}\varphi =\Delta \mu ,\quad\text{ in }\Omega \times \left(
0,\infty \right) ,  \label{nlch} \\
& \mu =a\varphi -J\ast \varphi +F^{\prime }\left( \varphi \right) +\alpha
\partial _{t}\varphi,\quad \text{ in }\Omega \times \left( 0,\infty \right) ,
\label{chem}
\end{align}%
subject to the no-flux boundary condition
\begin{equation}
\partial _{\mathbf{n}}\mu =0,\quad \text{ on }\Gamma \times \left( 0,\infty
\right)  \label{bc}
\end{equation}%
and to the initial condition
\begin{equation}
\varphi (0)=\varphi _{0},\quad \text{ in }\Omega .  \label{ic}
\end{equation}%
Here the coefficient $\alpha \geq 0$ characterizes the possible influences
of internal microforces (see, e.g., \cite{N-C1}). The presence of this term
is not necessary in the case of regular potentials, while it is crucial in
the case of singular ones. In fact, in the former case, in order to prove
our main result we need to first establish the eventual boundedness of $%
\varphi $. This boundedness is, say, built-in in the latter case, but we
need to show that $\varphi $ has the separation property uniformly with
respect to the initial data. This feature is an open problem even for the
classical \emph{local} Cahn-Hilliard equation with constant mobility in
dimension three (see \cite{MZ}).

The paper is organized as follows. Section 2 is devoted to the nonviscous
case with a regular potential, while Section 3 is concerned with the viscous
equation with a singular potential. Provided suitable global bounds are
obtained (this is the most technical part), the existence of an exponential
attractor is proven through a short trajectory type technique devised in
\cite{EZ}. We also show that, in both cases, each solution converges to a
single equilibrium by using a suitable version of the {\L }ojasiewicz-Simon
inequality, provided that $F$ is real analytic. In the final Section 4, we
consider the (nonviscous) equation with degenerate mobility and logarithmic
potential. On account of the validity of the separation property, we can
still prove the existence of an exponential attractor.

\section{The nonviscous case with regular potential}

\subsection{Some preliminary results}

We begin with some basic notation and assumptions. Let us first set $%
H:=L^{2}\left( \Omega \right) $ and $V:=H^{1}\left( \Omega \right) .$ For
every $\psi \in V^{\prime },$ $V^{\prime }$ the dual space of $V$, we denote
by $\left\langle \psi \right\rangle $ the average of $\psi $ over $\Omega $,
that is,%
\begin{equation*}
\left\langle \psi \right\rangle =\frac{1}{\left\vert \Omega \right\vert }%
\left\langle \psi ,1\right\rangle
\end{equation*}%
where $\left\vert \Omega \right\vert $ stands for the Lebesgue measure of $%
\Omega $ and $\left\langle \cdot ,\cdot \right\rangle $ is the duality
product. Then we introduce the spaces $V_{0}:=\left\{ \psi \in
V:\left\langle \psi \right\rangle =0\right\} $, $V_{0}^{\prime }:=\left\{
\psi \in V^{\prime }:\left\langle \psi ,1\right\rangle =0\right\} $, and the
operator $A_{N}:V\rightarrow V^{\prime }$, $A_{N}\in \mathcal{L}\left(
V,V^{\prime }\right) $, defined by,
\begin{equation*}
\left\langle A_{N}u,v\right\rangle =\int_{\Omega }\nabla u\cdot \nabla
v\,dx,\quad \forall u,v\in V.
\end{equation*}%
With these definitions, it is well known that $A_{N\mid V_{0}}$ maps $V_{0}$
into $V_{0}^{\prime }$ isomorphically, and that the inverse map $\mathcal{N}%
=A_{N}^{-1}:V_{0}^{\prime }\rightarrow V_{0}$, is defined by%
\begin{equation*}
A_{N}\mathcal{N}\psi =\psi ,\quad \forall \psi \in V_{0}^{\prime },\qquad
\mathcal{N}A_{N}f=f,\quad \forall f\in V_{0}.
\end{equation*}%
These maps also satisfy the following well-known relations:%
\begin{align}
\left\langle A_{N}u,\mathcal{N}v\right\rangle & =\left\langle
u,v\right\rangle ,\quad \forall u\in V,v\in V_{0}^{\prime },  \label{rel} \\
\left\langle u,\mathcal{N}v\right\rangle & =\left\langle v,\mathcal{N}%
u\right\rangle ,\quad \forall u,v\in V_{0}^{\prime }.  \notag
\end{align}%
The assumptions listed below are the same as in \cite{BH} (see also \cite%
{CFG}).

\begin{enumerate}
\item[(H1)] $J\in W^{1,1}\left( \mathbb{R}^{d}\right) $, $J\left( -x\right)
=J\left( x\right) $, $a\left( x\right) :=\int_{\Omega }J\left( x-y\right)
dy\geq 0$ a.e. in $\Omega .$

\item[(H2)] $F\in C_{\text{loc}}^{2,1}\left( \mathbb{R}\right) $ and there
exists $c_{0}>0$ such that%
\begin{equation*}
F^{\prime \prime}\left( s\right) +\inf_{x\in \Omega }a\left( x\right) \geq
c_{0}, \quad\forall s\in \mathbb{R}\text{.}
\end{equation*}

\item[(H3)] There exist $c_{1}>\frac{1}{2}\left\Vert J\right\Vert
_{L^{1}\left( \mathbb{R}^{d}\right) }$ and $c_{2}\in \mathbb{R}$ such that%
\begin{equation*}
F\left( s\right) \geq c_{1}s^{2}-c_{2},\quad\forall s\in \mathbb{R}\text{.}
\end{equation*}

\item[(H4)] There exist $c_{3}>0$, $c_{4}\geq 0$ and $p\in (1,2]$ such that%
\begin{equation*}
|F^\prime(s)|^{p}\leq c_{3}|F(s)|+c_{4},\quad \forall s\in \mathbb{R}.
\end{equation*}

\item[(H5)] $F\in C^{2}\left( \mathbb{R}\right) $ and there exist $%
c_{5},c_{6}>0$ and $q>0$ such that%
\begin{equation*}
F^{\prime \prime}(s)+\inf_{x\in \Omega }a(x)\geq
c_{5}s^{2q}-c_{6},\quad\forall s\in \mathbb{R}.
\end{equation*}
\end{enumerate}

\begin{remark}
\label{compact} Note that the operator $\psi \mapsto J*\psi$ is self-adjoint
and compact from $H$ to itself, provided that (H1) is satisfied. Also, it is
easy to realize that it is compact from $L^{\infty }(\Omega )$ to $C^{0}(%
\overline{\Omega })$ and that $a\in L^\infty(\Omega)$.
\end{remark}

We report the following result (see \cite[Corollary 1 and Proposition 5]{FG2}%
, cf. also \cite{CFG}).

\begin{theorem}
\label{well}Let $\varphi _{0}\in H$ with $F\left( \varphi _{0}\right) \in
L^{1}\left( \Omega \right) $, and assume (H1)-(H4), and (H5) with $q\geq
\frac{1}{2}$ and $p\in \left( \frac{6}{5},2\right] $ when $d=3$. Then, there
exists a unique weak solution of (\ref{nlch})-(\ref{ic}) on $\left[ 0,T%
\right] $, for any $T>0$, such that%
\begin{equation}
\begin{array}{l}
\varphi \in C\left( \left[ 0,T\right] ;H\right) \cap L^{2}\left( \left[ 0,T%
\right] ;V\right) \cap L^{\infty }\left( \left[ 0,T\right] ;L^{2+2q}\left(
\Omega \right) \right) , \\
\partial _{t}\varphi \in L^{2}\left( \left[ 0,T\right] ;V^{\prime }\right)
\text{, }\mu \in L^{2}\left( \left[ 0,T\right] ;V\right) , \\
F\left( \varphi \right) \in L^\infty\left(\left[ 0,T\right]
;L^{1}\left(\Omega \right)\right)%
\end{array}
\label{class}
\end{equation}%
and $\varphi \left( 0\right) =\varphi _{0}$, $\left\langle \varphi \left(
t\right) \right\rangle =\left\langle \varphi _{0}\right\rangle ,$ for all $%
t\in \left[ 0,T\right] $. Furthermore, setting%
\begin{equation}
\mathcal{E}\left( \varphi \left( t\right) \right) :=\frac{1}{4}\int_{\Omega
}\int_{\Omega }J\left( x-y\right) \left( \varphi \left( x,t\right) -\varphi
\left( y,t\right) \right) ^{2}dxdy+\int_{\Omega }F\left( \varphi \left(
x,t\right) \right) dx,  \label{ener}
\end{equation}%
the following equality holds for all $t\geq 0$,%
\begin{equation}
\mathcal{E}\left( \varphi \left( t\right) \right) +\int_{0}^{t}\left\Vert
\nabla \mu \left( s\right) \right\Vert _{H}^{2}ds=\mathcal{E}\left( \varphi
_{0}\right) .  \label{eneeq}
\end{equation}
\end{theorem}

\begin{remark}
\label{wf}We say that $\varphi $ is a weak solution of (\ref{nlch})-(\ref{ic}%
) on $(0,T)$ if $\varphi $ belongs to the class of functions (\ref{class}),
and for every $\psi \in V,$%
\begin{equation}
\langle \partial _{t}\varphi ,\psi \rangle +\left( \nabla \mu ,\nabla \psi
\right) =0,\quad \text{ a.e. in }\left( 0,T\right) ,  \label{weakf}
\end{equation}%
see \cite[Definition 1]{FG2}. Equivalently, we can write (\ref{weakf}) in
the following form:%
\begin{equation}
\langle \partial _{t}\varphi ,\psi \rangle +\left( \nabla \Phi ,\nabla \psi
\right) =\left( \nabla J\ast \varphi ,\nabla \psi \right),\quad \text{ a.e.
in }\left( 0,T\right) ,  \label{weakf2}
\end{equation}%
where $\Phi \left( x,\varphi \right) :=a\left( x\right) \varphi +
F^\prime\left( \varphi \right) $ (cf. \cite[Section 2]{BH}).
\end{remark}

\begin{remark}
It is not difficult to check that a Gaussian $J\left( x\right)
=c_{J}e^{-\xi\left\vert x\right\vert ^{2}} $, $\xi>0$, or a Newtonian
interaction kernel $J\left( x\right) =c_{J}\left\vert x\right\vert ^{-1},$
if $d=3$, $J\left( x\right) =-c_{J}\log \left( \left\vert x\right\vert
\right) $ if $d=2,$ fulfills (H1). Note that (H5) is slightly stronger than
(H2) and is necessary to establish the energy identity (\ref{eneeq}). The
further restriction on $p\in \left( \frac{6}{5},2\right] $ when $d=3$ is
also required for (\ref{eneeq}) to hold. Otherwise, we would only have an
energy inequality. On the other hand, (H1)-(H4) are enough to establish the
existence of at least one global weak solution (cf. \cite{CFG}). Observe
that assumption (H4) is fulfilled by a potential of arbitrary polynomial
growth. Besides, (H2)-(H5) are certainly satisfied, for instance, by %
\eqref{dwpot}.
\end{remark}

The next result can also be found in \cite[Theorem 1]{FG2}.

\begin{proposition}
Let $m\geq 0$ be given. Then every weak solution to (\ref{nlch})-(\ref{ic})
satisfies the dissipative estimate:%
\begin{equation}
\mathcal{E}\left( \varphi \left( t\right) \right) \leq \mathcal{E}\left(
\varphi _{0}\right) e^{-kt}+C\left( m\right) ,\text{ }\forall t\geq 0,
\label{dissi}
\end{equation}%
provided that $\left\vert \left\langle \varphi _{0}\right\rangle \right\vert
\leq m$, where $k$ and $C$ are positive constants independent of time and
initial data, but which depend on the other structural parameters of the
problem.
\end{proposition}

\begin{remark}
The proof of (\ref{dissi}) does not require the validity of the energy
identity (\ref{eneeq}), and so it holds also outside the range $p\in \left(
\frac{6}{5},2\right] $ when $d=3$, see \cite{FG1}.
\end{remark}

Let us now set%
\begin{equation*}
\mathcal{Y}_{m}:=\left\{ \psi \in H\,:\,F\left( \psi\right) \in L^{1}\left(
\Omega \right) ,\; \left\vert \left\langle \varphi \right\rangle \right\vert
\leq m\right\} ,
\end{equation*}%
and endow $\mathcal{Y}_{m}$ with the following metric%
\begin{equation}
d\left( \psi _{1},\psi _{2}\right) :=\left\Vert \psi
_{1}-\psi_{2}\right\Vert _{H} +\left\vert \int_{\Omega }F\left( \psi
_{1}\right) dx-\int_{\Omega }F\left( \psi _{2}\right) dx\right\vert ^{1/2},
\label{metric}
\end{equation}%
for any $\psi_{1},\psi_{2}\in \mathcal{Y}_{m}$. Thanks to (\ref{dissi}) and
Theorem \ref{well}, we can associate with problem (\ref{nlch})-(\ref{ic})
the solution semiflow%
\begin{equation}
S\left( t\right) :\mathcal{Y}_{m}\rightarrow \mathcal{Y}_{m},\quad \varphi
_{0}\mapsto S\left( t\right) \varphi _{0}=\varphi \left( t\right) ,
\label{sem}
\end{equation}%
where $\varphi \left( t\right) $ is the unique weak solution of (\ref{nlch}%
)-(\ref{ic}).

Concerning the long-term behavior, there holds (see \cite[Theorem 4]{FG2})

\begin{theorem}
\label{gl_at}Let the assumptions of Theorem \ref{well} be satisfied. The
dynamical system $\left( S\left( t\right) ,\mathcal{Y}_{m}\right) $
possesses a connected global attractor $\mathcal{A}$.
\end{theorem}

\subsection{Exponential attractors}

The main results of this section are contained in the following

\begin{theorem}
\label{expo} For every fixed $m\geq 0$, there exists an exponential
attractor $\mathcal{M}=\mathcal{M}\left( m\right) $ bounded in $V\cap
C^{\alpha }\left( \overline{\Omega }\right) $ for the dynamical system $%
\left( \mathcal{Y}_{m},S\left( t\right) \right) $ which satisfies the
following properties:

(i) Semi-invariance: $S\left( t\right) \mathcal{M}\subset \mathcal{M}$, for
every $t\geq 0.$

(ii) Exponential attraction:%
\begin{equation*}
dist_{C^{\alpha }\left( \overline{\Omega }\right) \cap H^{1-\nu }\left(
\Omega \right) }\left( S\left( t\right) \mathcal{Y}_{m},\mathcal{M}\right)
\leq Ce^{-\kappa t},\quad \forall t\geq 0,
\end{equation*}%
for some positive constants $C_{m}$ and $\kappa $, for any $\nu \in (0,1)$
and some $\alpha \in \left( 0,1\right) .$

(iii) Finite dimensionality:%
\begin{equation*}
\dim _{F}\left( \mathcal{M},C^{\alpha }\left( \overline{\Omega }\right)
\right) \leq C_{m,\alpha }<\infty .
\end{equation*}
\end{theorem}

Thus we can immediately deduce the

\begin{corollary}
The global attractor $\mathcal{A}$ is bounded in $V\cap C^{\alpha }\left(
\overline{\Omega }\right) $ and has finite fractal dimension:%
\begin{equation*}
\dim _{F}\left( \mathcal{A},C^{\alpha }\left( \overline{\Omega }\right)
\right) <\infty .
\end{equation*}
\end{corollary}

To prove Theorem \ref{expo} we first need to derive a number of properties
of the semigroup solution. The first result gives a dissipative estimate in
the space $L^{\infty }\left( \Omega \right) .$

\begin{lemma}
\label{linf} Let the assumptions of Theorem \ref{well} be satisfied. For
every $\tau >0$, there exists a constant $C_{m,\tau }>0$ such that%
\begin{equation}
\sup_{t\geq 2\tau }\left\Vert \varphi \left( t\right) \right\Vert
_{L^{\infty }\left( \Omega \right) }\leq C_{m,\tau }.  \label{linf1}
\end{equation}%
Moreover, there exists $R_{0}>0$ (independent of time, $\tau $ and initial
data) such that $S\left( t\right) $ possesses an absorbing ball $\mathcal{B}%
_{L^{\infty }\left( \Omega \right) }\left( R_{0}\right) $, bounded in $%
L^{\infty }\left( \Omega \right) $.
\end{lemma}

\begin{proof}
Our proof of (\ref{linf1}) relies on an iterative argument as in \cite{G0}.
The estimates will be derived assuming sufficiently smooth solutions to (\ref%
{nlch})-(\ref{ic}) so that the function $\left\vert \varphi \right\vert
^{p-1}\varphi $ is also $L^{2}$-summable for each $p>1.$ The scheme we
employ is as follows: let $\varphi _{0\varepsilon }\in L^{\infty }\left(
\Omega \right) $ such that $\varphi _{0\varepsilon }\rightarrow \varphi _{0}$
in $H,$ and such that $F\left( \varphi _{0\varepsilon }\right) \rightarrow
F\left( \varphi _{0}\right) $ in $L^{1}\left( \Omega \right) $ as $%
\varepsilon \rightarrow 0$. In this case, we can exploit the existence proof
of Theorem \ref{well} (see \cite{CFG}) one more time and an a priori $%
L^{\infty }$-estimate from \cite[Theorem 2.1]{BH} to deduce the existence of
a weak solution $\varphi _{\varepsilon }$ satisfying (\ref{class}) with the
additional \emph{essential} property%
\begin{equation}
\varphi _{\varepsilon }\in L^{\infty }\left( \mathbb{R}_{+}\times \Omega
\right), \quad\forall \varepsilon >0.  \label{suf}
\end{equation}%
Also for practical purposes, $C$ denotes from now on a positive constant
that is independent of $t$, $\varepsilon $, $\varphi$ and initial data, but
which only depends on the other structural parameters. Such a constant may
vary even from line to line. Further dependencies of this constant on other
parameters will be pointed out as needed.

For $p>1$, omitting the subscript $\varepsilon $, we multiply equation (\ref%
{nlch}) by $\left\vert \varphi \right\vert ^{p-1}\varphi $ and integrate
over $\Omega $, to obtain%
\begin{equation}
\frac{d}{dt}\int_{\Omega }\left\vert \varphi \right\vert ^{p+1}dx+\frac{2pC}{%
p+1}\int_{\Omega }\left\vert \nabla \left\vert \varphi \right\vert ^{\frac{%
p+1}{2}}\right\vert ^{2}dx\leq C\left( p+1\right) ^{2}\int_{\Omega
}\left\vert \varphi \right\vert ^{p+1}dx,  \label{lp}
\end{equation}%
where $C>0$ is independent of $p$ and $\varepsilon >0$, owing to the
assumptions (H1)-(H2) (cf. \cite[Theorem 2.1, (2.8)-(2.16)]{BH}). Note that
the regularity (\ref{suf}) is key in proving this estimate. Setting now $%
p=2^{k}-1$, $k\geq 0,$ then%
\begin{equation*}
x_{k}\left( t\right) :=\int_{\Omega }\left\vert \varphi \left( t\right)
\right\vert ^{2^{k}}dx,\quad k\geq 0,
\end{equation*}%
and having established (\ref{lp}), we can now exploit the scheme in \cite[%
Theorem 3.2, (3.8)-(3.10)]{G0} (see also \cite[Theorem 2.3]{G00}) to derive
the following inequality:%
\begin{equation}
x_{k}\left( t\right) \leq C_{\xi }\left( 2^{k}\right) ^{\sigma }\left(
\sup_{s\geq t-\xi /2^{k}}x_{k-1}\left( s\right) \right) ^{2},\quad \forall
k\geq 1,  \label{iti}
\end{equation}%
where $t,\xi $ are two positive constants such that $t-\xi /2^{k}>0$, and $%
C_{\xi },$ $\sigma $ are positive constants independent of $k$; the constant
$C_{\xi }$ is bounded if $\xi $ is bounded away from zero. We can iterate in
(\ref{iti}) reasoning exactly as in, e.g., \cite[Theorem 3.2]{G0} (cf.,
also, \cite[Theorem 2.3]{G00}). For the sake of completeness, we report a
sketch of the argument. Choose any numbers $\tau ^{\prime }>\tau >0$ such
that $\xi =(\tau ^{\prime }-\tau ),$ $t_{0}=\tau ^{\prime }$ and $%
t_{k}=t_{k-1}-\xi /2^{k},$ $k\geq 1$. Thus, in view of (\ref{iti}) we have%
\begin{equation}
\sup_{t\geq t_{k-1}}x_{k}\left( t\right) \leq C_{\xi }\left( 2^{k}\right)
^{\sigma }(\sup_{s\geq t_{k}}x_{k-1}\left( s\right) )^{2},\quad k\geq 1.
\label{e13}
\end{equation}%
Next, define%
\begin{equation}
C_{H}:=\sup_{s\geq t_{1}}x_{0}\left( s\right) =\sup_{s\geq t_{1}}\left\Vert
\varphi \left( s\right) \right\Vert _{H}^{2}.  \label{e13bis}
\end{equation}%
Thus, we can iterate in (\ref{e13}) with respect to $k\geq 1$ and obtain that%
\begin{equation}
\sup_{t\geq t_{0}}x_{k}\left( t\right) \leq \sup_{t\geq t_{k-1}}x_{k}\left(
t\right) \leq C_{\xi }^{A_{k}}2^{\sigma B_{k}}\left( C_{H}\right) ^{2^{k}},
\label{e14}
\end{equation}%
where%
\begin{equation}
A_{k}:=1+2+2^{2}+...+2^{k}\leq 2^{k}\sum\limits_{i\geq 1}\frac{1}{2^{i}}
\label{ak2}
\end{equation}%
\begin{equation}
B_{k}:=k+2\left( k-1\right) +2^{2}\left( k-2\right) +...+2^{k}\leq
2^{k}\sum\limits_{i\geq 1}\frac{i}{2^{i}}.  \label{bk2}
\end{equation}%
Therefore, taking the $2^{k}$-root on both sides of (\ref{e14}) and then
letting $k\rightarrow +\infty $ (note that the series in (\ref{ak2})-(\ref%
{bk2}) are convergent), we deduce%
\begin{equation}
\sup_{t\geq t_{0}=\tau ^{\prime }}\left\Vert \varphi \left( t\right)
\right\Vert _{L^{\infty }\left( \Omega \right) }\leq \lim_{k\rightarrow
+\infty }\sup_{t\geq t_{0}}\left( x_{k}\left( t\right) \right)
^{1/2^{k}}\leq C_{\xi }\left( C_{H}\right) ,  \label{l2linf}
\end{equation}%
for some positive constant $C_{0}$ independent of $t,$ $k$, $\varphi ,$ $%
\varepsilon ,$ $\xi $ and initial data.

In order to prove the first assertion of lemma, we observe that a simple
argument \cite[Proposition 4, (3.21)-(3.22)]{FG2} yields, on account of (\ref%
{dissi}), that%
\begin{equation}
C_{H}=\sup_{t\geq \frac{3\tau }{2}}\left\Vert \varphi \left( t\right)
\right\Vert _{H}^{2}\leq C_{m}\left( 1+\mathcal{E}\left( \varphi _{0}\right)
\right) .  \label{dis}
\end{equation}%
Thus, setting $\tau ^{\prime }=2\tau $ so that $\xi =\tau $, we readily
obtain the first claim (\ref{linf1})\ of lemma from (\ref{dis}). On the
other hand, the same argument as in \cite[Proposition 4, (3.21)-(3.22)]{FG2}
yields a bounded absorbing ball in $H$. Indeed, in light of (\ref{dissi}),
it is not difficult to see that, for any bounded set $\mathcal{B}\subset
\mathcal{Y}_{m}$, there exists a time $t_{\ast }=t_{\ast }\left( \mathcal{B}%
\right) >0$ such that $S\left( t\right) \mathcal{B}\subset H,$ for all $%
t\geq t_{\ast }$. Next, we can choose $\tau ^{\prime }=\tau +2\xi $ with $%
\tau =t_{\ast }$ and $\xi =1,$ so that $C_{H}$ and $C_{\xi }$ are bounded
uniformly with respect to initial data as $t\geq t_{\ast }$. Hence, the $%
L^{2}$-$L^{\infty }$ smoothing property (\ref{l2linf}) immediately entails
the second assertion of lemma.
\end{proof}

We also have

\begin{lemma}
\label{holderb} Let the assumptions of Theorem \ref{well} be satisfied.
Then, for every $\tau >0$, there exists a constant $C_{m,\tau ,\alpha }>0$
such that%
\begin{equation}
\sup_{t\geq 2\tau }\left\Vert \varphi \right\Vert _{C^{\alpha /2,\alpha
}\left( \left[ t,t+1\right] \times \overline{\Omega }\right) }\leq C_{m,\tau
,\alpha },  \label{holdest}
\end{equation}%
for some $\alpha \in \left( 0,1\right) .$ Thus, there exists $R_{1}>0$
(independent of time, $\tau $ and initial data) such that $S\left( t\right) $
possesses an absorbing ball $\mathcal{B}_{C^{\alpha }\left( \overline{\Omega
}\right) }\left( R_{1}\right) $, bounded in $C^{\alpha }\left( \overline{%
\Omega }\right) $.
\end{lemma}

\begin{proof}
We can rewrite the system (\ref{nlch})-(\ref{bc}) in the following form
\begin{equation}
\partial _{t}\varphi =\text{div}\left( a\left( x,\varphi ,\nabla \varphi
\right) \right) ,\text{ }\left( a\left( x,\varphi ,\nabla \varphi \right)
\cdot \mathbf{n}\right) _{\mid \Gamma }=0,  \label{sysnew1}
\end{equation}%
where%
\begin{equation}
a\left( x,\varphi ,\nabla \varphi \right) :=\left( a\left( x\right)
+F^{\prime \prime }\left( \varphi \right) \right) \nabla \varphi +\left(
\nabla a\right) \varphi -\nabla J\ast \varphi .
\end{equation}%
Since $J\in W^{1,1}(\mathbb{R}^{d})$ and $\varphi $ is bounded by Lemma \ref%
{linf}, using the fact that $a\left( x\right) +F^{\prime \prime }\left(
\varphi \right) \geq c_{0},$ by (H2), it is easy to check that%
\begin{equation*}
a\left( x,\varphi ,\nabla \varphi \right) \nabla \varphi \geq \frac{c_{0}}{2}%
\left\vert \nabla \varphi \right\vert ^{2}-C_{1},\text{ }\left\vert a\left(
x,\varphi ,\nabla \varphi \right) \right\vert \leq C_{2}\left\vert \nabla
\varphi \right\vert +C_{3},
\end{equation*}%
for some positive constants $C_{i}$ which depend only on $\left\Vert
J\right\Vert _{W^{1,1}}$ and (\ref{linf1}). Thus, the desired estimate in (%
\ref{holdest}) follows from the application of \cite[Corollary 4.2]{Du}. The
proof is finished.
\end{proof}

In view of Lemma \ref{linf} and the proof of \cite[Theorem 4.3, (4.19)-(4.39)%
]{BH}, the following result is now straightforward.

\begin{lemma}
\label{h1lemma}Let the assumptions of Theorem \ref{well} be satisfied. Then,
for every $\tau >0$, there exists a constant $C_{m,\tau }>0$ such that%
\begin{equation}
\sup_{t\geq 3\tau }\left[ \left\Vert \varphi \left( t\right) \right\Vert
_{V}+\left\Vert \partial _{t}\varphi \right\Vert _{L^{2}\left( \left[ t,t+1%
\right] \times \Omega \right) }\right] \leq C_{m,\tau }.  \label{h1est}
\end{equation}%
Moreover, for any bounded set $\mathcal{B}\subset \mathcal{Y}_{m}$, there
exists a time $t_{\#}=t_{\#}\left( \mathcal{B}\right) >0$ such that $S\left(
t\right) \mathcal{B}\subset V,$ for all $t\geq t_{\#}.$
\end{lemma}

The following result shows that the semigroup is strongly continuous with
respect to the $V^{\prime }$-metric.

\begin{proposition}
\label{uniq} Let $\varphi _{i},$ $i=1,2$, be a pair of weak solutions
according to the assumptions of Theorem \ref{well}. Then the following
estimate holds:
\begin{align}
& \left\Vert \varphi _{1}\left( t\right) -\varphi _{2}\left( t\right)
\right\Vert _{V^{\prime }}^{2}+c_{0}\int_{0}^{t}\left\Vert \varphi
_{1}\left( s\right) -\varphi _{2}\left( s\right) \right\Vert _{H}^{2}ds
\label{Lipschitz} \\
& \leq \left\Vert \varphi _{1}\left( 0\right) -\varphi _{2}\left( 0\right)
\right\Vert _{V^{\prime }}^{2}e^{\kappa t}+Ce^{\kappa t}\left\vert
M_{1}-M_{2}\right\vert ,  \notag
\end{align}%
for all $t\geq 0$, where $M_{i}:=\left\langle \varphi _{i}\left( 0\right)
\right\rangle $, for some positive constants $\kappa ,C$ which depend on $%
c_{0}$ and $J$ but are independent of $\varphi _{i}\left( 0\right) .$
\end{proposition}

\begin{proof}
We have that $\varphi :=\varphi _{1}-\varphi _{2}$ (formally) satisfies the
problem:
\begin{equation}
\partial _{t}\varphi =\Delta \overline{\mu },\text{ }\overline{\mu }=a\left(
x\right) \varphi -J\ast \varphi + F^\prime\left( \varphi _{1}\right)
-F^\prime\left( \varphi _{2}\right) ,  \label{d1}
\end{equation}%
subject to the boundary and initial conditions%
\begin{equation}
\partial _{\mathbf{n}}\overline{\mu }_{\mid \Gamma }=0\text{, }\varphi
_{\mid t=0}=\varphi _{1}\left( 0\right) -\varphi _{2}\left( 0\right) \text{
in }\Omega \text{.}  \label{d2}
\end{equation}%
Also, observe that (\ref{d1})-(\ref{d2}) yields $\left\langle \varphi \left(
t\right) \right\rangle =M_{1}-M_{2}$, for all $t\geq 0$. Consider now the
operator $A_{N}=-\Delta _{N}$, with domain $D\left( A_{N}\right) =\left\{
\varphi \in H^{2}\left( \Omega \right) :\partial _{\mathbf{n}}\varphi _{\mid
\Gamma }=0\right\} $. Test the first equation of (\ref{d1}) with $%
A_{N}^{-1}\left( \varphi \left( t\right) -\left\langle \varphi
(t)\right\rangle \right) $, then integrate by parts exploiting the relations
(\ref{rel}). We obtain, thanks to the assumptions (H1)-(H2) and arguing as
in \cite[Proposition 5, (4.2)-(4.3)]{FG2}, the following estimate:%
\begin{align}
&\frac{d}{dt}\left\Vert \varphi \left( t\right) \right\Vert _{V^{\prime
}}^{2}+2c_{0}\left\Vert \varphi \left( t\right) \right\Vert _{H}^{2}
\label{decay1b} \\
& \leq 2\left\vert \left( J\ast \varphi \left( t\right) ,\varphi \left(
t\right) \right) \right\vert +2\left\vert \left\langle \overline{\mu }\left(
t\right) \right\rangle \right\vert \left\vert \Omega \right\vert \left\vert
\left\langle \varphi \left( t\right) \right\rangle \right\vert  \notag \\
& \leq 2\left\Vert A_{N}^{1/2}\left( J\ast \varphi \left( t\right) \right)
\right\Vert _{H}\left\Vert A_{N}^{-1/2}\varphi \left( t\right) \right\Vert
_{H}  \notag \\
& +2\left\vert \left\langle \overline{\mu }\left( t\right) \right\rangle
\right\vert \left\vert \Omega \right\vert \left\vert M_{1}-M_{2}\right\vert
\notag \\
& \leq c_{0}\left\Vert \varphi \left( t\right) \right\Vert
_{H}^{2}+C\left\Vert \varphi \left( t\right) \right\Vert _{V^{\prime
}}^{2}+C\left\vert \left\langle \overline{\mu }\left( t\right) \right\rangle
\right\vert \left\vert M_{1}-M_{2}\right\vert ,  \notag
\end{align}%
for all $t\geq 0$. Applying Gronwall's inequality to (\ref{decay1b}) and
using the estimate%
\begin{equation}
\int_{0}^{t}\left\vert \left\langle \overline{\mu }\left( s\right)
\right\rangle \right\vert ds\leq C_{m}\left( 1+t\right) ,\text{ }\forall
t\geq 0  \label{chem_l1}
\end{equation}%
(this follows easily due to estimate (\ref{dissi}) and assumptions
(H3)-(H5)), we obtain estimate (\ref{Lipschitz}).
\end{proof}

The crucial step in order to establish the existence of an exponential
attractor is the validity of so-called smoothing property for the difference
of two solutions (see \cite{MZ}). In the present case, such a property is a
consequence of the following two lemmas. The first result establishes that
the semigroup $S\left( t\right) $ is some kind of contraction map, up to the
term $\left\Vert \varphi _{1}-\varphi _{2}\right\Vert _{L^{2}(\left[ 3\tau ,t%
\right] ;V^{\prime })}$.

\begin{lemma}
\label{dec} Let the assumptions of Proposition \ref{uniq} hold. Then, for
every $\tau >0$, we have:%
\begin{align}
& \left\Vert \varphi _{1}\left( t\right) -\varphi _{2}\left( t\right)
\right\Vert _{V^{\prime }}^{2}+C\left\vert M_{1}-M_{2}\right\vert
\label{decay1} \\
& \leq e^{-\kappa t}\left( \left\Vert \varphi _{1}\left( 0\right) -\varphi
_{2}\left( 0\right) \right\Vert _{V^{\prime }}^{2}+C\left\vert
M_{1}-M_{2}\right\vert \right)  \notag \\
& +C_{m,\tau }\int_{0}^{t}\left( \left\Vert \varphi _{1}\left( s\right)
-\varphi _{2}\left( s\right) \right\Vert _{V^{\prime }}^{2}+\left\vert
M_{1}-M_{2}\right\vert \right) ds,  \notag
\end{align}%
for all $t\geq 3\tau $, for some positive constants $C,C_{m,\tau \text{, }%
}\kappa $ which depend on $c_{0},$ $\Omega $ and $J.$
\end{lemma}

\begin{proof}
First, we observe that, due to the estimates (\ref{linf1}) and (\ref{h1est}%
), there holds:%
\begin{equation}
\sup_{t\geq 3\tau }\left\vert \left\langle \overline{\mu }\left( t\right)
\right\rangle \right\vert \leq C_{m,\tau },  \label{chem_sp}
\end{equation}%
for every (weak) solutions $\varphi _{1},\varphi _{2}$. Thus, combining (\ref%
{decay1b}) together with Poincar\'{e}'s inequality
\begin{equation*}
\left\Vert A_{N}^{-1/2}\left( \varphi -\left\langle \varphi \right\rangle
\right) \right\Vert _{H}^{2}+\left\langle \varphi \right\rangle ^{2}\leq
C_{\Omega }\left\Vert \varphi \right\Vert _{H}^{2},
\end{equation*}%
we deduce from (\ref{decay1b}) and (\ref{chem_sp}) the following inequality:%
\begin{align}
& \frac{d}{dt}\left( \left\Vert \varphi \left( t\right) \right\Vert
_{V^{\prime }}^{2}+C\left\vert M_{1}-M_{2}\right\vert \right) +c_{0}\left(
\left\Vert \varphi \left( t\right) \right\Vert _{V^{\prime
}}^{2}+C\left\vert M_{1}-M_{2}\right\vert \right)  \label{decay1bis} \\
& \leq C\left\Vert \varphi \left( t\right) \right\Vert _{V^{\prime
}}^{2}+C_{m,\tau }\left\vert M_{1}-M_{2}\right\vert ,  \notag
\end{align}%
for all $t\geq 3\tau $. Thus, Gronwall's inequality entails the desired
estimate (\ref{decay1}).
\end{proof}

We now need some compactness for the term $\left\Vert \varphi _{1}-\varphi
_{2}\right\Vert _{L^{2}\left( \left[ 3\tau ,t\right] ;V^{\prime }\right) }$
on the right-hand side of (\ref{decay1}). This is given by

\begin{lemma}
\label{lipdif} Let the assumptions of Proposition \ref{uniq} hold. Then, for
every $\tau >0$, the following estimate holds:%
\begin{align}
& \left\Vert \partial _{t}\varphi _{1}-\partial _{t}\varphi _{2}\right\Vert
_{L^{2}\left( \left[ 3\tau ,t\right] ;D\left( A_{N}\right) ^{^{\prime
}}\right) }^{2}+c_{0}\int_{0}^{t}\left\Vert \varphi _{1}\left( s\right)
-\varphi _{2}\left( s\right) \right\Vert _{H}^{2}ds  \label{comp1} \\
& \leq C_{m,\tau }e^{\kappa t}\left\Vert \varphi _{1}\left( 0\right)
-\varphi _{2}\left( 0\right) \right\Vert _{V^{\prime }}^{2}+Ce^{\kappa
t}\left\vert M_{1}-M_{2}\right\vert ,  \notag
\end{align}%
for all $t\geq 3\tau $, where $C_{m,\tau },C$ and $\kappa >0$ also depend on
$c_{0},$ $\Omega $ and $J.$
\end{lemma}

\begin{proof}
The second term on the left-hand side of (\ref{comp1}) can be easily
controlled by (\ref{Lipschitz}). Thus we only need to estimate the time
derivative. Recall that $\varphi $ satisfies (\ref{d1}). Furthermore, in
light of Lemmas \ref{linf} and \ref{h1lemma}, recall that we have%
\begin{equation}
\sup_{t\geq 3\tau }\left\Vert \varphi _{i}\left( t\right) \right\Vert
_{V\cap L^{\infty }\left( \Omega \right) }\leq C_{m,\tau }\text{, }i=1,2.
\label{linfbis}
\end{equation}%
Thus, for any test function $\psi \in D(A_{N})$, using the weak formulation (%
\ref{weakf}), there holds
\begin{align}
\left\langle \partial _{t}\varphi \left( t\right) ,\psi \right\rangle &
=\left( \nabla \overline{\mu }\left( t\right) ,\nabla \psi \right)
=\left\langle \overline{\mu }\left( t\right) ,\Delta _{N}\psi \right\rangle
\label{asr} \\
& \leq \left\Vert \overline{\mu }\left( t\right) \right\Vert _{H}\left\Vert
\psi \right\Vert _{D\left( A_{N}\right) }\leq C_{m,\tau }\left\Vert \varphi
\right\Vert _{H}\left\Vert \psi \right\Vert _{D\left( A_{N}\right) }.  \notag
\end{align}%
This estimate together with (\ref{Lipschitz}) gives the desired estimate on
the time derivative in (\ref{comp1}).
\end{proof}

We now show that the semigroup $S\left( t\right) $ is actually uniformly H%
\"{o}lder continuous in the $C^{\alpha }$-norm with respect to the initial
data.

\begin{lemma}
\label{holder} Let $\varphi _{i}\left( t\right) =S\left( t\right) \varphi
_{i}\left( 0\right) $, with $\varphi _{i}(0)\in \mathcal{Y}_{m}$ such that $%
M_{i}=\left\langle \varphi _{i}\left( 0\right) \right\rangle $, $i=1,2$.
Then, for every $\tau >0$, the following estimate is valid:%
\begin{equation}
\left\Vert \varphi _{1}\left( t\right) -\varphi _{2}\left( t\right)
\right\Vert _{C^{\alpha /2}\left( \overline{\Omega }\right) }\leq C_{m,\tau
}e^{\kappa t}\left( \left\Vert \varphi _{1}\left( 0\right) -\varphi
_{2}\left( 0\right) \right\Vert _{V^{\prime }}^{\beta }+\left\vert
M_{1}-M_{2}\right\vert ^{\frac{\beta }{2}}\right) ,  \label{hclinf}
\end{equation}%
for all $t\geq 3\tau $, where the constants $C_{m,\tau },$ $\kappa $ and $%
\beta <1$ are independent of the initial data.
\end{lemma}

\begin{proof}
Using the interpolation $[V,V^{\prime }]_{1/2,2}=H$, we deduce from
estimates (\ref{Lipschitz}) and (\ref{linfbis}) that
\begin{equation}
\left\Vert \varphi _{1}\left( t\right) -\varphi _{2}\left( t\right)
\right\Vert _{H}\leq C_{m,\tau }e^{\kappa t}\left( \left\Vert \varphi
_{1}\left( 0\right) -\varphi _{2}\left( 0\right) \right\Vert _{V^{\prime
}}^{1/2}+\left\vert M_{1}-M_{2}\right\vert ^{1/4}\right) ,  \label{hc2}
\end{equation}%
for all $t\geq 3\tau $. On the other hand, due to the boundedness of $%
\varphi _{i}\left( t\right) \in L^{\infty }\left( \Omega \right) \cap V$ for
$t\geq 3\tau $, $i=1,2$ (cf. \ref{linfbis}), the nonlinearity $f=F^{\prime }$
is controlled by the linear part of the equation (\ref{d1}) (no matter how
fast it grows) and obtaining the $L^{2}$-$L^{\infty }$ smoothing property
for our dynamical system is actually reduced to the same standard procedure
we used in the proof of Lemma \ref{linf}. Indeed, we already have an
estimate of the $L^{\infty }$-norm of the solution $\varphi \left( t\right) $
(due to (\ref{linf1})) and, consequently, we do not need to worry about the
growth of $f=F^{\prime }$. In particular, estimate (\ref{lp}) also holds for
the difference of solutions $\varphi =\varphi _{1}-\varphi _{2}$. This
observation combined with (\ref{hc2}) and a proper interpolation inequality
between $C^{\alpha }\left( \overline{\Omega }\right) \subset C^{\alpha
/2}\left( \overline{\Omega }\right) \subset L^{\infty }\left( \Omega \right)
$ implies the desired inequality (\ref{hclinf}).
\end{proof}

The last ingredient we need is the uniform H\"{o}lder continuity of $%
t\mapsto S(t)\varphi _{0}$ in the $C^{\alpha }$-norm, namely,

\begin{lemma}
\label{hctime} Let the assumptions of Theorem \ref{well} be satisfied.
Consider $\varphi \left( t\right) =S\left( t\right) \varphi _{0}$ with $%
\varphi _{0}\in \mathcal{Y}_{m}$. Then, for every $\tau >0$, there holds
\begin{equation}
\left\Vert \varphi \left( t\right) -\varphi \left( s\right) \right\Vert
_{C^{\alpha /2}\left( \overline{\Omega }\right) }\leq C_{m,\tau }\left\vert
t-s\right\vert ^{\beta },\quad \forall t,s\geq 3\tau ,  \label{hclinf2}
\end{equation}%
where $\beta <1$ and the positive constant $C_{m,\tau }$ is independent of
initial data, $\varphi $ and $t,s$.
\end{lemma}

\begin{proof}
According to (\ref{linfbis}), the following bound holds for $\mu $:
\begin{equation*}
\sup_{t\geq 3\tau }\left\Vert \mu \left( t\right) \right\Vert _{V}\leq
C_{m,\tau }.
\end{equation*}%
Consequently, by comparison in (\ref{nlch}), we have that
\begin{equation*}
\sup_{t\geq 3\tau }\left\Vert \partial _{t}\varphi \left( t\right)
\right\Vert _{V^{\prime }}\leq C_{m,\tau },
\end{equation*}%
which entails
\begin{equation}
\left\Vert \varphi \left( t\right) -\varphi \left( s\right) \right\Vert
_{V^{\prime }}\leq C_{m,\tau }\left\vert t-s\right\vert ,\quad \forall
t,s\geq 3\tau .  \label{hcb3}
\end{equation}%
Estimate (\ref{hclinf2}) now follows from (\ref{hcb3}), $\left[ V,V^{\prime }%
\right] _{1/2,2}=H$, the $L^{2}$-$L^{\infty }$ smoothing property of the
solutions in $\left[ 3\tau ,\infty \right) $ and the interpolation inequality%
\begin{equation}
\left\Vert \varphi \right\Vert _{C^{\alpha /2}\left( \overline{\Omega }%
\right) }\leq C\left\Vert \varphi \right\Vert _{C^{\alpha }\left( \overline{%
\Omega }\right) }^{1-\zeta }\left\Vert \varphi \right\Vert _{L^{\infty
}\left( \Omega \right) }^{\zeta },  \label{interpolation}
\end{equation}%
for some $\zeta =\zeta \left( \alpha \right) >0.$
\end{proof}

We report for the reader's convenience the following abstract result on the
existence of exponential attractors \cite[Proposition 4.1]{EZ} which will be
used in the following proof and in the other sections as well.

\begin{proposition}
\label{abstract}Let $\mathcal{H}$,$\mathcal{V}$,$\mathcal{V}_{1}$ be Banach
spaces such that the embedding $\mathcal{V}_{1}\subset \mathcal{V}$ is
compact. Let $B$ be a closed bounded subset of $\mathcal{H}$ and let $%
\mathbb{S}:B\rightarrow B$ be a map. Assume also that there exists a
uniformly Lipschitz continuous map $\mathbb{T}:B\rightarrow \mathcal{V}_{1}$%
, i.e.,%
\begin{equation}
\left\Vert \mathbb{T}b_{1}-\mathbb{T}b_{2}\right\Vert _{\mathcal{V}_{1}}\leq
L\left\Vert b_{1}-b_{2}\right\Vert _{\mathcal{H}},\quad \forall
b_{1},b_{2}\in B,  \label{gl1}
\end{equation}%
for some $L\geq 0$, such that%
\begin{equation}
\left\Vert \mathbb{S}b_{1}-\mathbb{S}b_{2}\right\Vert _{\mathcal{H}}\leq
\gamma \left\Vert b_{1}-b_{2}\right\Vert _{\mathcal{H}}+K\left\Vert \mathbb{T%
}b_{1}-\mathbb{T}b_{2}\right\Vert _{\mathcal{V}},\quad \forall
b_{1},b_{2}\in B,  \label{gl2}
\end{equation}%
for some $\gamma <\frac{1}{2}$ and $K\geq 0$. Then, there exists a
(discrete) exponential attractor $\mathcal{M}_{d}\subset B$ of the semigroup
$\{\mathbb{S}(n):=\mathbb{S}^{n},n\in Z+\}$ with discrete time in the phase
space $\mathcal{H}$.
\end{proposition}

\noindent \textit{Proof of Theorem \ref{expo}}. In order to apply
Proposition \ref{abstract}, it is sufficient to verify the existence of an
exponential attractor for the restriction of $S(t)$ on some properly chosen
semi-invariant absorbing set in $\mathcal{Y}_{m}$. Recall that, by Lemmas %
\ref{linf} and \ref{h1lemma}, the ball $B_{0}:=B_{C^{a}\left( \overline{%
\Omega }\right) \cap V}\left( R_{0}\right) $ will be absorbing for $S\left(
t\right) $, provided that $R_{0}>0$ is sufficiently large. Since we want
this ball to be semi-invariant with respect to the semigroup, we push it
forward by the semigroup, by defining first the set $B_{1}=\left[ \cup
_{t\geq 0}S\left( t\right) B_{0}\right] _{V_{m}^{\prime }}$, where $\left[
\cdot \right] _{V_{m}^{\prime }}$ denotes closure in the space
\begin{equation*}
V_{m}^{\prime }:=\left\{ \psi \in V^{\prime }\,:\,\left\vert \left\langle
\psi \right\rangle \right\vert \leq m\right\} ,
\end{equation*}%
and then the set $\mathbb{B}=S\left( 1\right) B_{1}$. Thus, $\mathbb{B}$ is
a semi-invariant compact (for the metric $d_{V_{m}^{^{\prime }}}\left(
\varphi _{1},\varphi _{2}\right) =\left\Vert \varphi _{1}-\varphi
_{2}\right\Vert _{V^{^{\prime }}}+\left\vert \left\langle \varphi
_{1}-\varphi _{2}\right\rangle \right\vert ^{1/2}$) subset of the phase
space $\mathcal{Y}_{m}$. On the other hand, due to the results proven in
this section, we have%
\begin{equation*}
\sup_{t\geq 0}\left( \left\Vert \varphi \left( t\right) \right\Vert
_{C^{a}\left( \overline{\Omega }\right) \cap V}+\left\Vert \mu \left(
t\right) \right\Vert _{V}+\left\Vert \partial _{t}\varphi \left( t\right)
\right\Vert _{V^{\prime }}\right) \leq C_{m},
\end{equation*}%
for every trajectory $\varphi $ originating from $\varphi _{0}\in \mathbb{B}$%
, for some positive constant $C_{m}$ which is independent of the choice of $%
\varphi _{0}\in \mathbb{B}$. We can now apply the abstract result above to
the map $\mathbb{S}=S\left( T\right) $ and $\mathcal{H}=V_{m}^{\prime }$,
for a fixed $T>0$ such that $e^{-\kappa T}<\frac{1}{2}$, where $\kappa >0$
is the same as in Lemma \ref{dec}. To this end, we introduce the functional
spaces%
\begin{equation}
\mathcal{V}_{1}:=L^{2}\left( \left[ 0,T\right] ;L^{2}\left( \Omega \right)
\right) \cap H^{1}\left( \left[ 0,T\right] ;D(A_{N})^{\prime }\right) ,\quad
\mathcal{V}:=L^{2}\left( \left[ 0,T\right] ;V_{m}^{\prime }\right) ,
\label{fsp}
\end{equation}%
and note that $\mathcal{V}_{1}$ is compactly embedded into $\mathcal{V}$.
Finally, we introduce the operator $\mathbb{T}:\mathbb{B}\rightarrow
\mathcal{V}_{1}$, by $\mathbb{T}\varphi _{0}:=\varphi \in \mathcal{V}_{1},$
where $\varphi $ solves (\ref{nlch})-(\ref{ic}) with $\varphi \left(
0\right) =\varphi _{0}\in \mathbb{B}$. We claim that the maps $\mathbb{S}$, $%
\mathbb{T}$, the spaces $\mathcal{H}$,$\mathcal{V}$,$\mathcal{V}_{1}$ thus
defined satisfy all the assumptions of Proposition \ref{abstract}. Indeed,
the global Lipschitz continuity (\ref{gl1}) of $\mathbb{T}$ is an immediate
corollary of Lemma \ref{lipdif}, and estimate (\ref{gl2}) follows from
estimate (\ref{decay1}). Therefore, due to Proposition \ref{abstract}, the
semigroup $\mathbb{S}(n)=S\left( nT\right) $ generated by the iterations of
the operator $\mathbb{S}:\mathbb{B\rightarrow B}$ possesses a (discrete)
exponential attractor $\mathcal{M}_{d}$ in $\mathbb{B}$ endowed by the
topology of $V_{m}^{\prime }$. In order to construct the exponential
attractor $\mathcal{M}$ for the semigroup $S(t)$ with continuous time, we
note that, due to Lemma \ref{uniq}, this semigroup is Lipschitz continuous
with respect to the initial data in the topology of $V_{m}^{\prime }$.
Moreover, by (\ref{hclinf}) and (\ref{hclinf2}) the map $\left( t,\varphi
_{0}\right) \mapsto S\left( t\right) \varphi _{0}$ is also uniformly H\"{o}%
lder continuous on $\left[ 0,T\right] \times \mathbb{B}$, where $\mathbb{B}$
is endowed with the metric topology of $V_{m}^{\prime }$. Hence, the desired
exponential attractor $\mathcal{M}$ for the continuous semigroup $S(t)$ can
be obtained by the standard formula%
\begin{equation}
\mathcal{M}=\bigcup_{t\in \left[ 0,T\right] }S\left( t\right) \mathcal{M}%
_{d}.  \label{st}
\end{equation}%
In order to finish the proof of the theorem, we only need to verify that $%
\mathcal{M}$ defined as above will be the exponential attractor for $S(t)$
restricted to $\mathbb{B}$ not only with respect to the $V_{m}^{\prime }$%
-metric, but also in with respect to a stronger metric. This is an immediate
corollary of the following facts: $\mathbb{B}$ is bounded in $V\cap
C^{\alpha }\left( \overline{\Omega }\right) $, the $L^{2}$-$C^{\alpha
}\left( \overline{\Omega }\right) $ smoothing property of the map $\varphi
_{0}\mapsto S\left( t\right) \varphi _{0}$, and the interpolation
inequalities given by (\ref{interpolation}) and
\begin{equation*}
\left\Vert u\right\Vert _{H^{1-\nu }\left( \Omega \right) }\leq
C_{s}\left\Vert u\right\Vert _{H}^{s}\left\Vert u\right\Vert
_{V}^{1-s},\quad \nu \in (0,1),
\end{equation*}%
for some $s=s\left( \nu \right) \in (0,1)$. Theorem \ref{expo} is now proved.

\begin{remark}
The methods used in this section can also be applied to other nonlocal
problems which have a variational structure similar to the nonlocal
Cahn-Hilliard equation. An interesting case (see \cite[Sec.~5]{BH} and its
references) is a model related to interacting particle systems with Kawasaki
dynamics, namely,
\begin{equation}
\left\{
\begin{array}{ll}
\partial _{t}\varphi =\Delta \left( \varphi -\tanh (\beta J\left( x\right)
\ast \varphi) \right) & \text{in }\Omega\times(0,\infty) , \\
\partial _{\mathbf{n}}\left( \varphi -\tanh (\beta J\left( x\right) \ast
\varphi) \right) =0 & \text{on }\Gamma\times(0,\infty),%
\end{array}%
\right.  \label{K}
\end{equation}%
for some constant $\beta $. In fact, in this case the $L^{2}$-$L^{\infty }$
smoothing property proven in Lemma \ref{linf} holds again regardless of the
value of $\beta $. The existence of an absorbing set in $V\cap L^{\infty
}\left( \Omega \right) $ for the solution map $\varphi _{0}\in H$ $\mapsto
\varphi \left( t\right) \in H$ of (\ref{K}) can be also established as in
\cite[Section 5]{BH}. Hence, the existence of an exponential attractor for
the dynamical system associated with (\ref{K}) can be proven arguing as in
Theorem \ref{expo}.
\end{remark}

\subsection{Convergence to a single equilibrium}

Let $\varphi $ be a weak solution to (\ref{nlch})-(\ref{ic}) according to
Theorem \ref{well}. In this section we aim to show that that the $\omega $%
-limit set,%
\begin{equation*}
\omega \left[ \varphi \right] =\left\{ \varphi _{\ast }:\exists
t_{n}\rightarrow \infty \text{ such that }\varphi \left( t_{n}\right)
\rightarrow \varphi _{\ast }\text{ in }H\right\}
\end{equation*}%
is a singleton, where $\varphi _{\ast }$ is a solution of the stationary
problem:%
\begin{equation}
\left\{
\begin{array}{ll}
a(x)\varphi _{\ast }-J\ast \varphi _{\ast }+F^{\prime }\left( \varphi _{\ast
}\right) =\mu _{\ast },\quad & \text{a.e. in }\Omega , \\
\mu _{\ast }=\text{constant}, \quad \left\langle \varphi _{\ast
}\right\rangle =\left\langle \varphi _{0}\right\rangle &
\end{array}%
\right.  \label{stat}
\end{equation}%
(see, e.g., \cite[Theorem 4.5]{BH}). We employ a generalized version of the {%
\L }ojasiewicz-Simon theorem proved in \cite[Theorem 6]{GG2} (cf. also \cite[%
Section 4]{LP}). The version that applies to our case is formulated in the
following.

\begin{lemma}
\label{LS_lemma} Let $J$ satisfy (H1) and $F\in C^{2}\left( \mathbb{R}%
\right) $ be a real analytic function satisfying (H2). Then, there exist
constants $\theta \in (0,\frac{1}{2}],$ $C>0,$ $\varepsilon >0$ such that
the following inequality holds:%
\begin{equation}
\left\vert \mathcal{E}\left( \varphi \right) -\mathcal{E}\left( \varphi
_{\ast }\right) \right\vert ^{1-\theta }\leq C\left\Vert \mu -\left\langle
\mu \right\rangle \right\Vert _{H},  \label{LSineq}
\end{equation}%
for all $\varphi \in L^{\infty }\left( \Omega \right) \cap \mathcal{Y}_{m}$
provided that $\left\Vert \varphi -\varphi _{\ast }\right\Vert _{H}\leq
\varepsilon .$
\end{lemma}

\begin{proof}
We will now apply the abstract result \cite[Theorem 6]{GG2} to the energy
functional $\mathcal{E}\left( \varphi \right) $, which according to (\ref%
{ener}) is the sum of a double-well potential and an interface energy term.
In contrast to this feature, we shall split $\mathcal{E}\left( \varphi
\right) $ into the sum of a convex (entropy) functional $\Phi :H=L^{2}\left(
\Omega \right) \rightarrow \mathbb{R}\cup \left\{ \infty \right\} $, with a
suitable effective domain, and a non-local interaction functional $\Psi
:H\rightarrow \mathbb{R}$. To this end, we define the lower-semicontinuous
and strongly convex functional%
\begin{equation*}
\Phi \left( \varphi \right) :=\left\{
\begin{array}{ll}
\int_{\Omega }\left( F\left( \varphi \right) +\frac{a\left( x\right) }{2}%
\varphi ^{2}\right) dx, & \text{if }\varphi \in L^{\infty }\left( \Omega
\right) \\
+\infty , & \text{otherwise,}%
\end{array}%
\right.
\end{equation*}%
with closed effective domain dom$\left( \Phi \right) =L^{\infty }\left(
\Omega \right) \cap \mathcal{Y}_{m}$, and the quadratic functional $\Psi
:H\rightarrow \mathbb{R}$, given by%
\begin{align*}
\Psi \left( \varphi \right) & :=\frac{1}{4}\int_{\Omega }\int_{\Omega
}J\left( x-y\right) \left( \varphi \left( x,t\right) -\varphi \left(
y,t\right) \right) ^{2}dxdy+\int_{\Omega }\frac{a\left( x\right) }{2}\left(
2\varphi -\varphi ^{2}\right) dx \\
& =\int_{\Omega }\left( -\frac{1}{2}\varphi \left( J\ast \varphi \right)
+a\left( x\right) \varphi \right) dx
\end{align*}%
(the last equality is a direct computation). We have that $\Phi $ is Fr\'{e}%
chet differentiable on any open subset $\overline{U}$ of $U_{m}:=\left\{
\psi \in L^{\infty }(\Omega ):\left\vert \left\langle \psi \right\rangle
\right\vert \leq m\right\} ,$ with Fr\'{e}chet derivative $D\Phi :\overline{U%
}\rightarrow L^{\infty }\left( \Omega \right) $ having the form%
\begin{equation*}
\left\langle D\Phi \left( \varphi \right) ,\xi \right\rangle =\int_{\Omega
}\left( F^{\prime }\left( \varphi \right) +a\left( x\right) \varphi \right)
\cdot \xi dx,
\end{equation*}%
for all $\varphi \in \overline{U}$ and $\xi \in L^{\infty }\left( \Omega
\right) $. The analyticity of $D\Phi $ as a mapping on $L^{\infty }\left(
\Omega \right) $ is standard and can be proved exactly as in, e.g., \cite[%
Theorem 5.1]{FIP}. Moreover, due to assumption (H2), there holds%
\begin{equation*}
\left\langle D\Phi \left( \varphi _{1}\right) -D\Phi \left( \varphi
_{2}\right) ,\varphi _{1}-\varphi _{2}\right\rangle \geq c_{0}\left\Vert
\varphi _{1}-\varphi _{2}\right\Vert _{H}^{2},
\end{equation*}%
for all $\varphi _{1},\varphi _{2}\in \overline{U}$, and%
\begin{equation*}
\left\Vert D\Phi \left( \varphi _{1}\right) -D\Phi \left( \varphi
_{2}\right) \right\Vert _{H^{\ast }}\leq \gamma \left\Vert \varphi
_{1}-\varphi _{2}\right\Vert _{H},
\end{equation*}%
for some positive constant $\gamma .$ Moreover, computing the second Fr\'{e}%
chet derivative $D^{2}\Phi $\ of $\Phi $,%
\begin{equation*}
\left\langle D^{2}\Phi \left( \varphi \right) \xi _{1},\xi _{2}\right\rangle
=\int_{\Omega }\left( F^{\prime \prime }\left( \varphi \right) +a\left(
x\right) \right) \xi _{1}\cdot \xi _{2}dx
\end{equation*}%
yields that $D^{2}\Phi \in \mathcal{L}\left( L^{\infty }\left( \Omega
\right) ,L^{\infty }\left( \Omega \right) \right) $ is an isomorphism for
every $\varphi \in \overline{U}$. Concerning the (quadratic) function $\Psi $%
, we see that%
\begin{equation*}
\Psi \left( \varphi \right) =\frac{1}{2}\left\langle -J\ast \varphi ,\varphi
\right\rangle +\left\langle a\left( x\right) ,\varphi \right\rangle ,\text{ }%
a\in L^{\infty }\left( \Omega \right) \text{, }\forall \varphi \in H.
\end{equation*}%
We recall that the linear operator $\psi \mapsto J\ast \psi $ is
self-adjoint and compact from $H$ to itself and is also compact from $%
L^{\infty }(\Omega )$ to $C^{0}(\bar{\Omega})$ (cf. Remark \ref{compact}).
On the other hand, we also have the following (orthogonal) sum decomposition
of $H=H_{0}\oplus H_{1}$, where%
\begin{equation*}
H_{0}:=\left\{ \varphi \in H:\left\langle \varphi \right\rangle =0\right\}
\text{, }H_{1}:=\left\{ \varphi \in H:\varphi =\text{const}\right\}
\end{equation*}%
Thus the annihilator of $H_{0}$ is the one-dimensional subspace of constant
functions $H_{0}^{0}:=\left\{ ch\in H^{\ast }:c\in \mathbb{R}\right\}$,
where $h\in H^{\ast }\simeq H$ is given by $\left\langle h,u\right\rangle =$
$\frac{1}{\left\vert \Omega \right\vert }\int_{\Omega }udx$, $u\in H.$
Hence, the hypotheses of \cite[Theorem 6]{GG2} are satisfied and the sum%
\begin{equation*}
\mathcal{E}=\Phi +\Psi :H\rightarrow \mathbb{R}\cup \left\{ \infty \right\}
\end{equation*}%
is a well defined, bounded from below functional with nonempty, closed, and
convex effective domain dom$\left( F\right) =$dom$\left( \Phi \right) .$
Unravelling notation in \cite[Theorem 6]{GG2}, and observing that the Fr\'{e}%
chet derivative
\begin{equation*}
D\mathcal{E}\left( \varphi \right) =a(x)\varphi -J\ast \varphi +F^{\prime
}\left( \varphi \right) =\mu ,
\end{equation*}%
we have%
\begin{align*}
\left\vert \mathcal{E}\left( \varphi \right) -\mathcal{E}\left( \varphi
_{\ast }\right) \right\vert ^{1-\theta }& \leq C\inf_{\varphi \in H}\left\{
\left\Vert D\mathcal{E}\left( \varphi \right) -\mu _{\ast }\right\Vert
_{L^{2}\left( \Omega \right) }:\mu _{\ast }\in H_{0}^{0}\right\} \\
& =C\left\Vert \mu -\left\langle \mu \right\rangle \right\Vert _{L^{2}\left(
\Omega \right) },
\end{align*}%
from which (\ref{LSineq}) follows.
\end{proof}


We can now prove the following convergence result.

\begin{theorem}
\label{conv_thm}Let the assumptions of Theorem \ref{well} hold and, in
addition, assume that $F$ is real analytic. Then, any weak solution $\varphi
$ to problem (\ref{nlch})-(\ref{ic}) belonging to the class (\ref{class})
satisfies%
\begin{equation}
\lim_{t\rightarrow \infty }\left\Vert \varphi \left( t\right) -\varphi
_{\ast }\right\Vert _{L^{\infty }\left( \Omega \right) }=0,
\label{convlinftime}
\end{equation}%
where $\varphi _{\ast }$ is solution to (\ref{stat}).
\end{theorem}

\begin{proof}
Before we begin the proof, we note that by virtue of Lemma \ref{linf} and
Lemma \ref{h1lemma}, all $\varphi _{\ast }\in \omega \left[ \varphi \right] $
are bounded in $L^{\infty }\left( \Omega \right) \cap V$. Besides, recalling
also the energy identity (\ref{ener}), we have
\begin{equation*}
\mathcal{E}\left( \varphi \left( t\right) \right) \rightarrow \mathcal{E}%
_{\infty },\text{ as }t\rightarrow \infty ,
\end{equation*}%
and the limit energy $\mathcal{E}_{\infty }$ is the same for every
steady-state solution $\varphi _{\ast }\in $ $\omega \left[ \varphi \right] $%
. Moreover, we can integrate the energy equality (\ref{ener}),%
\begin{equation*}
\frac{d}{dt}\mathcal{E}\left( \varphi \left( t\right) \right) =-\left\Vert
\nabla \mu \left( t\right) \right\Vert _{H}^{2}
\end{equation*}%
over $\left( t,\infty \right) $ to get
\begin{equation}
\int_{t}^{\infty }\left\Vert \nabla \mu \left( s\right) \right\Vert
_{H}^{2}ds=\mathcal{E}\left( \varphi \left( t\right) \right) -\mathcal{E}%
_{\infty }=\mathcal{E}\left( \varphi \left( t\right) \right) -\mathcal{E}%
\left( \varphi _{\ast }\right) .  \label{convls0}
\end{equation}%
By virtue of Lemma \ref{LS_lemma} (cf. also Remark \ref{compact}), we have%
\begin{equation*}
\left\vert \mathcal{E}\left( \varphi \left( t\right) \right) -\mathcal{E}%
\left( \varphi _{\ast }\right) \right\vert ^{1-\theta }\leq C\left\Vert \mu
\left( t\right) -\left\langle \mu \left( t\right) \right\rangle \right\Vert
_{L^{2}\left( \Omega \right) }\leq C\left\Vert \nabla \mu \left( t\right)
\right\Vert _{H}^{2}
\end{equation*}%
provided that%
\begin{equation}
\left\Vert \varphi -\varphi _{\ast }\right\Vert _{H}\leq \varepsilon .
\label{small}
\end{equation}%
This, combined with the previous identity, yields%
\begin{equation}
\int_{t}^{\infty }\left\Vert \nabla \mu \left( s\right) \right\Vert
_{H}^{2}ds\leq C\left\Vert \nabla \mu \left( t\right) \right\Vert _{H}^{%
\frac{1}{\left( 1-\theta \right) }},  \label{convls1}
\end{equation}%
for all $t>0$, for as long as (\ref{small}) holds. Note that, in general,
the quantities $\theta ,$ $C$ and $\varepsilon $ above may depend on $%
\varphi _{\ast }$. Let us set%
\begin{equation*}
M=\cup \left\{ \mathcal{I}:\mathcal{I}\text{ is an open interval on which (%
\ref{small}) holds}\right\} .
\end{equation*}%
Clearly, $M$ is nonempty since $\varphi _{\ast }\in $ $\omega \left[ \varphi %
\right] $. We can now use (\ref{convls1}), the fact that $\left\Vert \nabla
\mu \left( t\right) \right\Vert _{H}\in L^{2}\left( 0,\infty \right) $ (cf. (%
\ref{ener})), and exploit \cite[Lemma 7.1]{FS} with $\alpha =2\left(
1-\theta \right) $ to deduce that $\left\Vert \nabla \mu \left( \cdot
\right) \right\Vert _{H}\in L^{1}\left( M\right) $ and%
\begin{equation}
\int_{M}\left\Vert \nabla \mu \left( s\right) \right\Vert _{H}ds\leq C\left(
\varphi _{\ast }\right) <\infty .  \label{convls2bis}
\end{equation}%
Consequently, using the bound (\ref{convls2bis}) and the main equation (\ref%
{nlch}), we also obtain%
\begin{equation}
\int_{M}\left\Vert \partial _{t}\varphi \left( s\right) \right\Vert
_{V^{^{\prime }}}ds<\infty .  \label{convls2tris}
\end{equation}%
In order to finish the proof of the convergence result in (\ref{convlinftime}%
) it suffices to show that it holds in $H$-norm. Indeed, in this case (\ref%
{convlinftime}) will become an immediate consequence of the $L^{2}$-$%
(L^{\infty }\cap V)$ smoothing property of the weak solutions and all $%
\varphi _{\ast }\in \omega \left[ \varphi \right] $ (see Lemmas \ref{linf}
and \ref{h1lemma}). We claim that we can find a sufficiently large time $%
\tau >0$ such that $\left( \tau ,\infty \right) \subset M$. To this end,
recalling (\ref{convls0}) and the above bounds, we also have that $\partial
_{t}\varphi \in L^{2}(0,\infty ;V^{^{\prime }})$, $\nabla \mu \in
L^{2}(0,\infty ;H^{d})$ and, furthermore, for any $\delta >0$ there exists a
time $t_{\ast }=t_{\ast }\left( \delta \right) >0$ such that
\begin{equation}
\left\Vert \partial _{t}\varphi \right\Vert _{L^{1}(M\cap \left( t_{\ast
},\infty \right) ;V^{^{\prime }})}\leq \delta ,\text{ }\left\Vert \partial
_{t}\varphi \right\Vert _{L^{2}(\left( t_{\ast },\infty \right) ;V^{^{\prime
}})}\leq \delta ,\text{ }\left\Vert \nabla \mu \right\Vert _{L^{2}(\left(
t_{\ast },\infty \right) ;H^{d})}\leq \delta .  \label{convls3}
\end{equation}%
Next, observe that by Lemma \ref{h1lemma} and Lemma \ref{linf}, there is a
time $t_{\#}>0$ such that%
\begin{equation}
\sup_{t\geq t_{\#}}\left\Vert \varphi \left( t\right) \right\Vert _{V\cap
L^{\infty }\left( \Omega \right) }\leq C.  \label{bls}
\end{equation}%
Now, let $\left( t_{0},t_{2}\right) \subset M$, for some $t_{2}>t_{0}\geq
t_{\ast }\left( \delta \right) ,$ $\left\vert t_{0}-t_{2}\right\vert \geq 1$
such that (\ref{bls}) holds (w.l.o.g., we shall assume that $t_{\ast }\geq
t_{\#}$). This claim is an immediate consequence of the aforementioned $%
L^{2} $-$(L^{\infty }\cap V)$ smoothing property and bounds (\ref{convls3}).
Using (\ref{convls3}) and (\ref{bls}), we obtain%
\begin{align}
\left\Vert \varphi \left( t_{0}\right) -\varphi \left( t_{2}\right)
\right\Vert _{H}^{2}& =2\int_{t_{0}}^{t_{2}}\left\langle \partial
_{t}\varphi \left( s\right) ,\varphi \left( s\right) -\varphi
\left(t_{0}\right) \right\rangle ds  \label{bls2b} \\
& \leq 2\int_{t_{0}}^{t_{2}}\left\Vert \partial _{t}\varphi \left( s\right)
\right\Vert _{V^{^{\prime }}}\left( \left\Vert \varphi \left( s\right)
\right\Vert _{V}+\left\Vert \varphi \left( t_{0}\right) \right\Vert
_{V}\right) ds  \notag \\
& \leq C\left\Vert \partial _{t}\varphi \right\Vert
_{L^{1}(t_{0},t_{2};V^{^{\prime }})}\left( \left\Vert \varphi \right\Vert
_{L^{\infty }(t_{\ast },\infty ;V)}+1\right) \leq C\delta .  \notag
\end{align}%
Therefore we can choose a time $t_{\ast }\left( \delta \right) =\tau
<t_{0}<t_{2}$, such that%
\begin{equation}
\left\Vert \varphi \left( t_{0}\right) -\varphi \left( t_{2}\right)
\right\Vert _{H}<\frac{\varepsilon }{3}  \label{bls3}
\end{equation}%
provided that (\ref{small}) holds for all $t\in \left( t_{0},t_{2}\right) $.
Since $\varphi _{\ast }\in \omega \left[ \varphi \right] $, a large
(redefined) $\tau $ can be chosen such that
\begin{equation}
\left\Vert \varphi \left( \tau \right) -\varphi _{\ast }\right\Vert
_{H}<\varepsilon /3;  \label{bls4}
\end{equation}%
hence, (\ref{bls3}) yields $\left( \tau ,\infty \right) \subset M$. Indeed,
taking%
\begin{equation*}
\overline{t}=\inf \left\{ t>\tau :\left\Vert \varphi \left( t\right)
-\varphi _{\ast }\right\Vert _{H}\geq \varepsilon \right\} ,
\end{equation*}%
we have $\overline{t}>\tau $ and $\left\Vert \varphi \left( \overline{t}%
\right) -\varphi _{\ast }\right\Vert _{H}\geq \varepsilon $ if $\overline{t}$
is finite. On the other hand, in view of (\ref{bls3}) and (\ref{bls4}), we
have%
\begin{equation*}
\left\Vert \varphi \left( t\right) -\varphi _{\ast }\right\Vert _{H}\leq
\left\Vert \varphi \left( t\right) -\varphi \left( \tau \right) \right\Vert
_{H}+\left\Vert \varphi \left( \tau \right) -\varphi _{\ast }\right\Vert
_{H}<\frac{2}{3}\varepsilon ,
\end{equation*}%
for all $\overline{t}>t\geq \tau $, and this leads to a contradiction.
Therefore, $\overline{t}=\infty $ and by (\ref{convls3}) the integrability
of $\partial _{t}\varphi $ in $L^{1}(\tau ,\infty ;V^{^{\prime }})$ follows.
Hence, $\omega \left[ \varphi \right] =\{\varphi _{\ast }\}$ and %
\eqref{convlinftime} holds on account of the $L^{2}$-$(L^{\infty }\cap V)$
smoothing property. The proof is finished.
\end{proof}

\begin{remark}
\label{rem_regp}Exploiting the $L^{2}$-$(C^{\alpha }\cap V)$ smoothing
property of the weak and stationary solutions again, and the inequality (\ref%
{LSineq}) it is be possible the show the convergence rate:%
\begin{equation*}
\left\Vert \varphi \left( t\right) -\varphi _{\ast }\right\Vert
_{C^{a/2}\left( \overline{\Omega }\right) }\sim \left( 1+t\right) ^{-\frac{1%
}{\rho }},\text{ as }t\rightarrow \infty ,
\end{equation*}%
for some positive constant $\rho =\rho \left( \alpha ,\theta ,\varphi _{\ast
}\right) \in \left( 0,1\right) .$
\end{remark}

\section{The viscous case with singular potential}

\subsection{Well-posedness result}

Here we consider the viscous case, but we suppose that the potential is
singular. More precisely, following \cite{FG1}, we assume that $F$ can be
written as $F=F_{1}+F_{2},$ where $F_{1}\in C^{(2+2q)}\left( -1,1\right)$,
with $q$ a fixed positive integer and $F_{2}\in C^{2}[-1,1]$. Such functions
are subject to the following hypotheses:

\begin{enumerate}
\item[(H7)] There exist $c_{1}>0$ and $\epsilon _{0}>0$ such that
\begin{equation*}
F_{1}^{(2+2q)}(s)\geq c_{1},\qquad \forall s\in (-1,-1+\epsilon _{0}]\cup
\lbrack 1-\epsilon _{0},1).
\end{equation*}

\item[(H8)] There exists $\epsilon _{0}>0$ such that, for each $k=0,1,\cdots
,2+2q$ and each $j=0,1,\cdots ,q$,
\begin{align*}
& F_{1}^{(k)}(s)\geq 0,\qquad \forall s\in \lbrack 1-\epsilon _{0},1), \\
& F_{1}^{(2j+2)}(s)\geq 0,\qquad F_{1}^{(2j+1)}(s)\leq 0,\qquad \forall s\in
(-1,-1+\epsilon _{0}].
\end{align*}

\item[(H9)] There exists $\epsilon _{0}>0$ such that $F_{1}^{(2+2q)}$ is
non-decreasing in $[1-\epsilon _{0},1)$ and non-increasing in $%
(-1,-1+\epsilon _{0}]$.

\item[(H10)] There exist $\alpha ,\beta \in \mathbb{R}$ with $\displaystyle %
\alpha +\beta >-\min_{[-1,1]}F_{2}^{\prime \prime }$ such that
\begin{equation*}
F_{1}^{\prime \prime }(s)\geq \alpha ,\qquad \forall s\in (-1,1),\quad
a(x)\geq \beta ,\text{ a.e. }x\in \Omega .
\end{equation*}

\item[(H11)] $\lim_{s\rightarrow \pm 1}F^{\prime }(s)=\pm \infty .$
\end{enumerate}

Assumptions (H7)-(H11) are satisfied, e.g., when $F$ is of the form %
\eqref{logpot}, for any fixed positive integer $q$. In particular, setting
\begin{equation*}
F_{1}(s)= (1+s)\log (1+s)+(1-s)\log (1-s),\qquad F_{2}(s)=- \lambda s^{2},
\end{equation*}%
assumption (H10) is satisfied if and only if $\beta >\gamma -1 $.

The notion of weak solution to problem (\ref{nlch})-(\ref{ic}) is given by

\begin{definition}
\label{weaksing}Let $0<T<+\infty $ be given. Suppose $\varphi _{0}\in
L^{\infty }\left( \Omega \right) ,$ $\left\vert \left\langle \varphi
_{0}\right\rangle \right\vert <1$ with $F(\varphi _{0})\in L^{1}(\Omega )$.
A function $\varphi $ is a weak solution of (\ref{nlch})-(\ref{ic}) on $%
[0,T] $, if%
\begin{align}
& \varphi \in L^{\infty }(0,T;H)\cap L^{\infty }(0,T;L^{2+2q}(\Omega ));
\label{re4} \\
& \partial _{t}\varphi \in L^{2}(0,T;V^{\prime }),\quad \sqrt{\alpha }%
\partial _{t}\varphi \in L^{2}(0,T;H);  \label{re5} \\
& \mu =a\varphi -J\ast \varphi +F^{\prime }(\varphi )+\alpha \partial
_{t}\varphi \in L^{2}(0,T;V);  \label{prmu} \\
& \varphi \in L^{\infty }(Q),\quad |\varphi (x,t)|<1\quad \text{a.e. }%
(x,t)\in Q:=\Omega \times (0,T);  \label{re6}
\end{align}%
and $\varphi $ satisfies the weak formulation (\ref{weakf}).
\end{definition}

The following existence result holds.

\begin{theorem}
\label{well_sing} Suppose that $J$ obeys (H1) and assume that (H7)-(H11) are
satisfied for some fixed positive integer $q$. Let $\varphi _{0}\in
L^{\infty }(\Omega )$ be such that $\left\vert \left\langle \varphi
_{0}\right\rangle \right\vert <1$ and $F(\varphi _{0})\in L^{1}(\Omega )$.
Then, for every $T>0$, there exists a weak solution $\varphi $ of (\ref{nlch}%
)-(\ref{ic}) in the sense of Definition \ref{weaksing}. In addition, for all
$t\geq 0$, we have $(\varphi (t),1)=(\varphi _{0},1)$ and the energy
identity holds:%
\begin{equation}
\mathcal{E}\left( \varphi \left( t\right) \right) +\int_{0}^{t}\left(
\left\Vert \nabla \mu \left( s\right) \right\Vert _{H}^{2}+\alpha \left\Vert
\partial _{t}\varphi \left( s\right) \right\Vert _{H}^{2}\right) ds=\mathcal{%
E}\left( \varphi _{0}\right) ,\text{ }\forall t\geq 0,  \label{idesing}
\end{equation}%
where $\mathcal{E}\left( \varphi \right) $ is given by (\ref{ener}).
\end{theorem}

\begin{proof}
The argument for $\alpha =0$ was given in \cite[Theorem 1 and Corollary 1]%
{FG1}. In the case $\alpha >0$ the proof goes essentially as in \cite{FG1}
with some minor modifications (see \emph{Step 1} below). Indeed, the whole
idea is to approximate (\ref{nlch})-(\ref{ic}) by a problem $P_{\epsilon }$\
which is obtained from (\ref{nlch})-(\ref{ic}) by replacing the singular
potential $F$ with a smooth (regular)\ potential of polynomial growth $%
F_{\epsilon }=F_{1\epsilon }+\overline{F}_{2},$ where $F_{1\epsilon }$ is
defined in such a way (cf. \cite[Lemma 1 and Lemma 2]{FG1}) that $%
F_{\epsilon }^{\prime }\rightarrow F^{\prime }$ uniformly on every compact
interval included in $(-1,1)$, and such that the following properties hold:

\noindent (i) There exist $c_{q},d_{q}>0$, which depend on $q$ but are
independent of $\epsilon $, and $\epsilon _{0}>0$ such that
\begin{equation}
F_{\epsilon }(s)\geq c_{q}|s|^{2+2q}-d_{q},\qquad \forall s\in \mathbb{R}%
,\quad \forall \epsilon \in (0,\epsilon _{0}].  \label{lb5}
\end{equation}%
\noindent (ii) Setting $\displaystyle c_{0}:=\alpha +\beta
+\min_{[-1,1]}F_{2}^{\prime \prime }>0$, there exists $\epsilon _{1}>0$ such
that
\begin{equation}
F_{\epsilon }^{\prime \prime }(s)+a(x)\geq c_{0},\qquad \forall s\in \mathbb{%
R},\quad \text{a.e. }x\in \Omega ,\quad \forall \epsilon \in (0,\epsilon
_{1}].  \label{lb6}
\end{equation}%
The approximating problem $P_{\epsilon },$ $\epsilon >0$, then takes the
following form: find a weak solution $\varphi _{\epsilon }$ to
\begin{equation}
\left\{
\begin{array}{ll}
\partial _{t}\varphi _{\epsilon }=\Delta \widetilde{\mu }_{\epsilon }, \quad
\widetilde{\mu }_{\epsilon }= \mu _{\epsilon }+\alpha \partial _{t}\varphi
_{\epsilon },\quad \text{ in }\Omega \times \left( 0,T\right), &  \\
\mu _{\epsilon } =a\varphi _{\epsilon }-J\ast \varphi _{\epsilon
}+F_{\epsilon }^{\prime }(\varphi _{\epsilon }), \quad \text{ in }\Omega
\times \left( 0,T\right) &  \\
\partial _{\mathbf{n}}\widetilde{\mu }_{\epsilon }=0, \quad \text{ on }%
\Gamma \times \left( 0,T\right) , &  \\
\varphi _{\epsilon }(0)=\varphi _{0},\quad \text{ in }\Omega. &
\end{array}
\right.  \label{approx}
\end{equation}
\noindent \emph{Step 1.} We will now briefly explain how to deduce the
existence of at least one weak solution to problem $P_{\epsilon }$ for a
given $\epsilon >0$. We first take an initial datum $\varphi _{0}$ such that
$\left\vert \left\langle \varphi _{0}\right\rangle \right\vert <1$, $%
F_{\epsilon }\left( \varphi _{0}\right) \in L^{1}\left( \Omega \right) $ and
\begin{equation}
\varphi _{0}\in L^{\infty }\left( \Omega \right) \cap V.  \label{ap_id}
\end{equation}
We also approximate the interaction kernel $J$ with, say, $J\in W^{1,\infty}(%
\mathbb{R}^d)$. 
To prove the existence of a solution $\varphi _{\epsilon }\left( t\right) $
to problem (\ref{approx}) (cf. Remark \ref{wf}) corresponding to the initial
datum $\varphi _{0}$ we can perform the same estimates as in \cite[ proof of
Theorem \ref{well}]{CFG} (cf. also \cite[Section 3]{FG1}) using a
Galerkin-type argument. Note that, this argument is actually independent of
whether $\alpha >0$ or $\alpha =0,$ since $\left\langle \partial _{t}\varphi
\left( t\right) \right\rangle =0$, for all $t\geq 0$, so that $\left\langle
\widetilde{\mu }_{\epsilon }\left( t\right) \right\rangle =\left\langle \mu
_{\epsilon }\left( t\right) \right\rangle $, for all $t\geq 0$. Thus, using
the assumptions (H8)-(H9) on $F$ and exploiting (i)-(ii) above, one can
argue exactly word by word as in \cite[(3.20)-(3.38)]{FG1} to deduce from (%
\ref{idesing}), for every $T>0,$ the following estimates on the Galerkin
approximating solutions (indices are omitted)
\begin{align}
\Vert \varphi _{\epsilon }\Vert _{L^{\infty }(0,T;L^{2+2q}(\Omega ))}& \leq
C,  \label{est1} \\
\Vert \nabla \widetilde{\mu }_{\epsilon }\Vert _{L^{2}(0,T;H)}& \leq C,
\notag \\
\sqrt{\alpha }\Vert \partial _{t}\varphi _{\epsilon }\Vert _{L^{2}(0,T;H)}&
\leq C,  \notag \\
||F_{\epsilon }^{\prime }(\varphi _{\epsilon })\Vert _{L^{\infty }\left(
0,T;L^{1}(\Omega )\right) }& \leq C,  \notag
\end{align}%
%
%
%
%
%
%
%
%
%
%
%
for some positive constant $C$\ which depends on the initial data $\varphi
_{0}\in L^{\infty }\left( \Omega \right) $, but is independent of $t,$ $T$, $%
\epsilon ,$ and $\alpha \geq 0$. These estimates and the Poincar\'{e}%
-Wirtinger inequality entail:%
\begin{equation}
\Vert \widetilde{\mu }_{\epsilon }\Vert _{L^{2}(0,T;V)}\leq C.  \label{est2}
\end{equation}%
Hence, in view of (\ref{est2}), by comparison in (\ref{approx}) we also have
the bound:%
\begin{equation}
\Vert \partial _{t}\varphi _{\epsilon }\Vert _{L^{2}(0,T;V^{\prime })}\leq C.
\label{est3}
\end{equation}

Next, we need another estimate for $\varphi _{\epsilon }\left( t\right) $ in
the $V$-norm. For this we can choose $\psi =\varphi _{\epsilon }\left(
t\right) $\ as a test function in (\ref{weakf}). We obtain%
\begin{align*}
& \frac{d}{dt}\left( \left\Vert \varphi _{\epsilon }\left( t\right)
\right\Vert _{H}^{2}+\alpha \left\Vert \nabla \varphi _{\epsilon }\left(
t\right) \right\Vert _{H}^{2}\right) \\
& +2\left( \left( a\left( x\right) +F_{\epsilon }^{\prime \prime }\left(
\varphi _{\epsilon }\right) \right) \nabla \varphi _{\epsilon }\left(
t\right) ,\nabla \varphi _{\epsilon }\left( t\right) \right) \\
& =2\left( \nabla J\ast \varphi _{\epsilon }\left( t\right) -\varphi
_{\epsilon }\left( t\right) \nabla a,\nabla \varphi _{\epsilon }\left(
t\right) \right) .
\end{align*}%
Recalling (ii) and estimate (\ref{est1}), we get
\begin{equation}
\frac{d}{dt}\left( \left\Vert \varphi _{\epsilon }\left( t\right)
\right\Vert _{H}^{2}+\alpha \left\Vert \nabla \varphi _{\epsilon }\left(
t\right) \right\Vert _{H}^{2}\right) +c_{0}\left\Vert \nabla \varphi
_{\epsilon }\left( t\right) \right\Vert _{H}^{2}\leq C_J,  \label{cda_pre}
\end{equation}%
for some appropriate constant $C_J>0$ which is independent of $\epsilon $
but depends on $\Vert \nabla J \Vert_{L^\infty(\mathbb{R}^d)}$. Estimate (%
\ref{cda_pre}) yields on account of a suitable Gronwall's inequality,
\begin{equation}
\left\Vert \varphi _{\epsilon }\left( t\right) \right\Vert _{V}^{2}\leq
\left\Vert \varphi _{0}\right\Vert _{V}^{2}e^{-\gamma t}+C_{\alpha,J},\quad
\forall t\geq 0,  \label{est4bis}
\end{equation}%
for some positive constant $\gamma $ independent of $\epsilon $. This
further estimate ensures that we have strong convergence in $L^2(0,T;H)$ for
some subsequence so that we can identify the nonlinear term in the
continuous limit. The existence of a solution to $P_{\epsilon }$ is proven
in the case of a smooth initial datum.

We can now establish the existence of a solution with an initial datum $%
\varphi_0\in H$ such that $\left\vert \left\langle \varphi _{0}\right\rangle
\right\vert <1$ and $F_{\epsilon }\left( \varphi _{0}\right) \in L^{1}\left(
\Omega \right) $ with a sequence $\varphi _{0j }\subset V$ with the same
properties and such that
\begin{equation}
\varphi _{0j }\rightarrow \varphi _{0}\text{ in }H\text{-norm,}
\label{ap_id2}
\end{equation}%
as $j \rightarrow \infty$. Also we take a sequence $\{J_j\} \in W^{1,\infty}(%
\mathbb{R}^d)$ which satisfies (H1) and strongly converges to $J$ in $%
W^{1,1}(\mathbb{R}^d)$ as $j \rightarrow \infty$.

Let $\{\varphi _{j,\epsilon }\}$ be a sequence of solutions associated with $%
\{\varphi _{j,0}\}$. Arguing as in the proof of Lemma \ref{unising} below,
we can get the estimate
\begin{align}
& \left\Vert \varphi _{j,\epsilon }\left( t\right) -\varphi _{i,\epsilon
}\left( t\right) \right\Vert _{V^{\prime }}^{2}+\alpha \left\Vert \varphi
_{j,\epsilon }\left( t\right) -\varphi _{i,\epsilon }\left( t\right)
\right\Vert _{H}^{2}  \label{cda} \\
& \leq \left( \left\Vert \varphi _{j,\epsilon }\left( 0\right) -\varphi
_{i,\epsilon }\left( 0\right) \right\Vert _{V^{\prime }}^{2}+\alpha
\left\Vert \varphi _{j,\epsilon }\left( 0\right) -\varphi _{i,\epsilon
}\left( 0\right) \right\Vert _{H}^{2}\right) e^{\kappa t}  \notag \\
& +C\left\vert \left\langle \varphi _{j,\epsilon }\left( 0\right)
\right\rangle -\left\langle \varphi _{i,\epsilon }\left( 0\right)
\right\rangle \right\vert e^{\kappa t},  \notag
\end{align}%
for some positive constants $C,$ $\kappa $ which depend on $\Vert
J_j\Vert_{W^{1,1}(\mathbb{R}^d)}$ and $\Omega$, but are independent of $%
\epsilon $ and $i,j$. This yields, on account of \eqref{ap_id2}, that as $%
j\rightarrow \infty ,$ we have the strong convergence of the sequence of
solutions $\varphi _{j,\epsilon }$ for $\alpha>0$, to some function $\varphi
_{\epsilon }$, i.e.,
\begin{equation}
\varphi _{j,\epsilon }\rightarrow \varphi _{\epsilon }\text{ strongly in }%
C\left( [0,T];H\right),  \label{est4}
\end{equation}
for every $\epsilon >0$. Finally, from the preceding estimates (\ref{est1})-(%
\ref{est3}), we can infer that (up to subsequences)
\begin{equation}
\left\{
\begin{array}{ll}
\partial _{t}\varphi _{j,\epsilon } & \rightarrow \partial _{t}\varphi
_{\epsilon }\text{ weakly in }L^{2}(0,T;H), \\
\varphi _{j,\epsilon } & \rightarrow \varphi _{\epsilon }\text{ weakly-star
in }L^{\infty }(0,T;L^{2+2q}(\Omega )), \\
\widetilde{\mu }_{j,\epsilon } & \rightarrow \widetilde{\mu }_{\epsilon }%
\text{ weakly in }L^{2}(0,T;V),%
\end{array}
\right.  \label{est5}
\end{equation}
as $j\rightarrow \infty $. Therefore, arguing as in \cite{CFG}, the
convergence properties (\ref{est4})-(\ref{est5}) allow us to show that $%
\varphi _{\epsilon }$ is a weak solution to problem $P_{\epsilon },$ with $%
\varphi _{0}$ satisfying the assumptions of Theorem \ref{well_sing} and $J$
fulfilling (H1).

\noindent \emph{Step 2.} The passage to limit as $\epsilon \rightarrow 0$ is
actually easier. Indeed, we have already observed that estimates (\ref{est1}%
)-(\ref{est3}) and (\ref{cda}) hold with constants independent of $\epsilon
>0$. The strong convergence can still be deduced by using the continuous
dependence estimate proven here below (see \eqref{Lipschitz_sing}). Thus a
similar argument works when $\epsilon $ goes to zero. We will only mention
that in order to pass to the limit in the variational formulation for
problem $P_{\epsilon },$ we need to show that $\left\vert \varphi
\right\vert <1$ almost everywhere in $Q=\Omega \times \left( 0,T\right) $.
This can be done by adapting an argument from \cite[Section 4]{DD}. We refer
the reader to \cite[Section 3]{FG1} for further details.
\end{proof}

Uniqueness is an immediate consequence of the following result whose proof
goes essentially as in Lemma \ref{dec}.

\begin{lemma}
\label{unising} Let $\varphi _{i},$ $i=1,2$, be a pair of weak solutions
according to the assumptions of Theorem \ref{well_sing}. Then the following
estimate holds:
\begin{align}
& \left\Vert \varphi _{1}\left( t\right) -\varphi _{2}\left( t\right)
\right\Vert _{V^{\prime }}^{2}+\alpha \left\Vert \varphi _{1}\left( t\right)
-\varphi _{2}\left( t\right) \right\Vert _{H}^{2}  \label{Lipschitz_sing} \\
& \leq \left( \left\Vert \varphi _{1}\left( 0\right) -\varphi _{2}\left(
0\right) \right\Vert _{V^{\prime }}^{2}+\alpha \left\Vert \varphi _{1}\left(
0\right) -\varphi _{2}\left( 0\right) \right\Vert _{H}^{2}\right) e^{\kappa
t} + Ce^{\kappa t}\left\vert M_{1}-M_{2}\right\vert ,  \notag
\end{align}%
for all $t\geq 0$, where $M_{i}:=\left\langle \varphi _{i}\left( 0\right)
\right\rangle $, for some positive constants $C,\kappa $ which depend on $%
c_{0}$ and $J$, but are independent of $\alpha \geq 0.$
\end{lemma}

\begin{proof}
We see that $\varphi $ (formally)\ satisfies the problem:%
\begin{equation}
\partial _{t}\varphi =\Delta \overline{\mu },\quad \overline{\mu }=a\left(
x\right) \varphi -J\ast \varphi + F^\prime\left( \varphi _{1}\right) -
F^\prime\left( \varphi _{2}\right) +\alpha \partial _{t}\varphi ,
\label{d1sing}
\end{equation}%
subject to the boundary and initial conditions%
\begin{equation}
\partial _{\mathbf{n}}\overline{\mu }_{\mid \Gamma }=0\text{, }\varphi
_{\mid t=0}=\varphi _{1}\left( 0\right) -\varphi _{2}\left( 0\right) \text{
in }\Omega \text{.}  \label{d2sing}
\end{equation}%
Arguing as in Lemma \ref{dec}, we obtain, on account of (H10), the following
estimate:%
\begin{align}
& \frac{d}{dt}\left( \left\Vert \varphi \left( t\right) \right\Vert
_{V^{\prime }}^{2}+\alpha \left\Vert \varphi \left( t\right) \right\Vert
_{H}^{2}\right) +2c_{0}\left\Vert \varphi \left( t\right) \right\Vert
_{H}^{2}  \label{diff_sing2} \\
& \leq c_{0}\left\Vert \varphi \left( t\right) \right\Vert _{H}^{2}+\kappa
\left\Vert \varphi \left( t\right) \right\Vert _{V^{\prime
}}^{2}+C\left\vert \left\langle \overline{\mu }\left( t\right) \right\rangle
\right\vert \left\vert M_{1}-M_{2}\right\vert ,  \notag
\end{align}%
where $\kappa ,C$ depend on $c_{0},$ $\Omega $ and $J$, but are independent
of $\alpha \geq 0$. Observe now that we also have $F^\prime\left( \varphi
\right) \in L^{\infty }\left( 0,T;L^{1}\left( \Omega \right) \right).$
Therefore we can still deduce (\ref{chem_l1}) and the application of
Gronwall's inequality to (\ref{diff_sing2}) entails the desired estimate (%
\ref{Lipschitz_sing}) exactly as in Lemma \ref{dec}.
\end{proof}

On account of the previous results, we can define a dynamical system on the
metric space
\begin{equation*}
\mathcal{Y}_{m}:=\{\psi \in L^{\infty }(\Omega )\,:\, \vert \psi \vert <1,
\text{ a.e. in } \Omega,\; F(\psi )\in L^{1}(\Omega ),\;|\left\langle \psi
\right\rangle |\leq m\}
\end{equation*}%
where $m\in [0,1)$ is fixed and the metric is given by (\ref{metric}). Then,
for each $\alpha \geq 0$ we can also define a semigroup%
\begin{equation}
S\left( t\right) :\mathcal{Y}_{m}\rightarrow \mathcal{Y}_{m}, \quad
S\left(t\right) \varphi _{0}=\varphi \left( t\right) ,  \label{sem_sing}
\end{equation}%
where $\varphi \left( t\right) $ is the unique weak solution of (\ref{nlch}%
)-(\ref{ic}). In fact, arguing as in \cite[Section 4, Theorem 2]{FG1}, we
deduce the following

\begin{theorem}
\label{gl_sing} Let the assumptions of Theorem \ref{well_sing} hold and
assume that $F$ is bounded in $(-1,1)$. Then the dynamical system $(\mathcal{%
Y}_{m},S(t))$ possesses a connected global attractor $\mathcal{A}$.
\end{theorem}

\subsection{Exponential attractors}

Note that, according to Theorem \ref{gl_sing}, a global attractor $\mathcal{A%
}$\ exists for any $\alpha \geq 0$. However, we are able to show its finite
dimensionality only in the case $\alpha >0$. This assumption is intimately
connected with the aforementioned separation property which will allow to
handle $F^{\prime }$ on a closed interval of the form $[-1+\delta,1-\delta]$%
. We have the following.

\begin{theorem}
\label{expo_2} Let the assumptions of Theorem \ref{well_sing} be satisfied.
If $\alpha>0$ then, for every fixed $m\geq 0$, there exists an exponential
attractor $\mathcal{M}=\mathcal{M}\left( m\right)$, bounded in $L^{\infty
}\left( \Omega \right) $ and compact in $H$, for the dynamical system $\left(%
\mathcal{Y}_{m}, S\left( t\right)\right) $ which satisfies the following
properties:

(i) Semi-invariance: $S\left( t\right) \mathcal{M}\subset \mathcal{M}$, for
every $t\geq 0.$

(ii) Separation property: there exists $\delta _{0}=\delta _{0}\left(
m,\alpha \right) \in \left( 0,1\right) $ such that
\begin{equation*}
\left\Vert \mathcal{M} \right\Vert _{L^{\infty }\left( \Omega \right) }\leq
1-\delta _{0}.
\end{equation*}

(iii) Exponential attraction:%
\begin{equation*}
dist_{L^{s}\left( \Omega \right) }\left( S\left( t\right) \mathcal{Y}_{m},%
\mathcal{M}\right) \leq Ce^{-\kappa t},
\end{equation*}%
for some positive constants $C_{m}$ and $\kappa $, for any $s\in \lbrack
2,\infty ).$

(iv) Finite dimensionality:%
\begin{equation*}
\dim _{F}\left( \mathcal{M},H\right) \leq C_{m}<\infty .
\end{equation*}
\end{theorem}

\begin{remark}
Note that, thanks to the separation property, the assumption that $F$ is
bounded on $(-1,1)$ (see Theorem \ref{gl_sing}) is not needed.
\end{remark}

\begin{corollary}
Let the assumptions of Theorem \ref{expo_2} be satisfied. The global
attractor $\mathcal{A}$ has finite fractal dimension%
\begin{equation*}
\dim _{F}\left( \mathcal{A},H\right) <\infty
\end{equation*}%
and satisfies $\left\Vert \mathcal{A}\right\Vert _{L^{\infty }\left( \Omega
\right) }\leq 1-\delta _{0}$, for some $\delta _{0}=\delta _{0}\left(
m,\alpha \right) \in \left( 0,1\right)$.
\end{corollary}

First, we derive some (uniform in time) a priori estimates for the weak
solutions. For the next result, we also assume that the boundary $\Gamma $
is of class $\mathcal{C}^{2}$ (we note that Lemma \ref{chem_lemma} is the
only place where this assumption is used; however, see Remark \ref{noreg}).

\begin{lemma}
\label{chem_lemma}Let the assumptions of Theorem \ref{expo_2} be satisfied.
For every $\tau >0$, there exists a positive constant $C_{m,\tau }\sim
1+\tau ^{-n_{0}}$ (for some $n_{0}>0$) such that the following estimate
holds:%
\begin{equation}
\sup_{t\geq \tau }\left( \left\Vert \mu \left( t\right) \right\Vert
_{H^{2}\left( \Omega \right) }+\alpha \left\Vert \partial _{t}\varphi \left(
t\right) \right\Vert _{H}\right) \leq C_{m,\tau }.  \label{chemp}
\end{equation}
\end{lemma}

\begin{proof}
To rigorously prove (\ref{chemp}), one has to employ the regularization
procedure introduced in Theorem \ref{well_sing} and to exploit the fact that
all the estimates below hold \emph{uniformly} in $\epsilon >0$ (we can also
employ a Faedo-Galerkin scheme for (\ref{approx}) to ensure that the
approximate solutions $\varphi _{\epsilon }$ are smooth enough).

To this end, set $\zeta :=\partial _{t}\varphi $ and note that $\left\langle
\zeta \left( t\right) \right\rangle =0$, for all $t\geq 0.$ According to (%
\ref{weakf}), the function $\zeta $ satisfies the following weak formulation:%
\begin{equation}
\left\langle \partial _{t}\zeta ,\psi \right\rangle +(\nabla \eta ,\nabla
\psi )=0,\text{ a.e. in }\left( 0,T\right) ,  \label{wd1}
\end{equation}%
for every $\psi \in V,$ where%
\begin{equation}
\eta =\left( a\left( x\right) +F^{\prime \prime }\left( \varphi \right)
\right) \zeta -J\ast \zeta +\alpha \partial _{t}\zeta ,\text{ a.e. in }%
\Omega \times \left( 0,T\right) .  \label{wd2}
\end{equation}%
As we mentioned above, we note that (\ref{wd1}) is actually intended to be
satisfied by a standard Galerkin approximation of $\varphi _{\varepsilon }$,
in which we should have at least $\partial _{t}\zeta \in L^{2}\left(
0,T;L^{2}\left( \Omega \right) \right) $. The required regularity in (\ref%
{chemp}) will be then obtained by passing to the limit in the subsequent
estimates. Thus, in what follows we shall proceed formally. Testing (\ref%
{wd1}) with $\psi =\mathcal{N}\zeta (=A_{N}^{-1}\zeta )$, then integrating
by parts, we obtain%
\begin{align*}
\frac{1}{2}\frac{d}{dt}\left\Vert \zeta \right\Vert _{V^{\prime }}^{2}&
=-\left( \eta -\left\langle \eta \right\rangle ,\zeta \right) _{H} \\
& =-\left( a\left( x\right) +F^{\prime \prime }\left( \varphi \right) ,\zeta
^{2}\right) _{H}+\left( J\ast \zeta ,\zeta \right) _{H}-\alpha \left(
\partial _{t}\zeta ,\zeta \right) _{H},
\end{align*}%
which yields, thanks to assumptions (H1) and (H10),
\begin{equation}
\frac{d}{dt}\left( \left\Vert \zeta \right\Vert _{V^{\prime }}^{2}+\alpha
\left\Vert \zeta \right\Vert _{H}^{2}\right) \leq C\left\Vert \zeta
\right\Vert _{H}^{2},  \label{ests1}
\end{equation}%
for some positive constant $C$ which depends only on $c_{0}$ and $J$. Thus,
using this inequality and exploiting the basic energy identity (\ref{idesing}%
), we have%
\begin{equation}
\sup_{t\geq 0}\int_{t}^{t+1}\left( \left\Vert \zeta \left( s\right)
\right\Vert _{H}^{2}+\left\Vert \mu \left( s\right) \right\Vert
_{V}^{2}\right) ds\leq C_{\alpha }\sim \alpha ^{-1}.  \label{ests2}
\end{equation}%
Thus, in view of the uniform Gronwall's lemma, we infer
\begin{equation}
\sup_{t\geq \tau }\left( \left\Vert \zeta \right\Vert _{V^{\prime
}}^{2}+\alpha \left\Vert \zeta \right\Vert _{H}^{2}\right) \leq \mathcal{C}=%
\mathcal{C}_{\alpha ,\tau }.  \label{wip}
\end{equation}%
From this point on, the constant $\mathcal{C}$ will always denote a
computable quantity whose expression is allowed to vary on occurrence,
depending on the initial data, on $\alpha ^{-1}>0,$ and on the other fixed
parameters of the system. We shall again point out its dependence on various
parameters whenever necessary. Therefore, by comparison in (\ref{nlch}) and
on account of (\ref{wip}), we deduce that
\begin{equation}
\sup_{t\geq \tau }\left\Vert \Delta \mu \left( t\right) \right\Vert
_{H}^{2}\leq \mathcal{C}_{\alpha ,\tau }.  \label{ests5}
\end{equation}%
Next, let us test (\ref{weakf}) by $\psi =\mathcal{N}\left( F^{\prime
}\left( \varphi \right) -\left\langle F^{\prime }\left( \varphi \right)
\right\rangle \right) $ to obtain
\begin{equation*}
\left\langle F^{\prime }\left( \varphi \right) -\left\langle F^{\prime
}\left( \varphi \right) \right\rangle ,\mathcal{N}\partial _{t}\varphi
\right\rangle =-\left\langle \mu ,F^{\prime }\left( \varphi \right)
-\left\langle F^{\prime }\left( \varphi \right) \right\rangle \right\rangle .
\end{equation*}%
Then note that%
\begin{align*}
\left\langle \mu ,F^{\prime }\left( \varphi \right) -\left\langle F^{\prime
}\left( \varphi \right) \right\rangle \right\rangle & =\left\langle a\varphi
-J\ast \varphi +F^{\prime }(\varphi )-\left\langle F^{\prime }\left( \varphi
\right) \right\rangle ,F^{\prime }(\varphi )-\left\langle F^{\prime }\left(
\varphi \right) \right\rangle \right\rangle \\
& \geq \frac{1}{2}\left\Vert F^{\prime }\left( \varphi \right) -\left\langle
F^{\prime }\left( \varphi \right) \right\rangle \right\Vert _{H}^{2}-\frac{1%
}{2}\Vert a\varphi -J\ast \varphi \Vert _{H}^{2} \\
& \geq \frac{1}{2}\left\Vert F^{\prime }\left( \varphi \right) -\left\langle
F^{\prime }\left( \varphi \right) \right\rangle \right\Vert
_{H}^{2}-C_{J}\Vert \varphi \Vert _{H}^{2}.
\end{align*}%
Therefore, on account of (\ref{idesing}) and (\ref{wip}), the above estimate
allows us to infer
\begin{equation*}
\left\Vert F^{\prime }\left( \varphi \right) -\left\langle F^{\prime }\left(
\varphi \right) \right\rangle \right\Vert _{H}\leq C(\Vert \mathcal{%
N\partial }_{t}\varphi \Vert _{H}+1)\leq \mathcal{C}_{\alpha ,\tau }.
\end{equation*}%
We can now easily argue as in the proof of Theorem \ref{well_sing}, see (\ref%
{est1}) (cf. \cite[Section 3, Theorem 1]{FG1} and \cite[Proposition A.2]{MZ}%
) to deduce%
\begin{equation*}
\sup_{t\geq \tau }\left\vert \left\langle \mu \left( t\right) \right\rangle
\right\vert ^{2}\leq \mathcal{C}_{\alpha ,\tau }.
\end{equation*}%
This estimate together with (\ref{wip}) and (\ref{ests5}) yield the desired
inequality (\ref{chemp}).
\end{proof}

\begin{remark}
\label{noreg}The assumption on $\Gamma \in \mathcal{C}^{2}$ can be dispensed
with so that the result below in Lemma \ref{sep} also holds for bounded
domains with \emph{Lipschitz} boundary $\Gamma $. Indeed, on account of
known elliptic regularity theory (cf., e.g., \cite{Da}) for problem (\ref%
{nlch}), (\ref{bc}), we can deduce that $\mu \left( t\right) \in L^{\infty
}\left( [\tau ,\infty );H^{1+\gamma }\left( \Omega \right) \right) ,$ for
any $\gamma \in \left( \frac{1}{2},1\right) $. Note that we cannot take $%
\gamma =1$ without further assumptions on $\Gamma $ (see \cite{Da}). Since $%
\Omega \subset \mathbb{R}^{d}$, $d\leq 3$, we have $H^{1+\gamma }\left(
\Omega \right) \subset L^{\infty }\left( \Omega \right) $ in the range
provided for $\gamma $ and the argument below in (3.36) still applies. Thus,
we can conclude the validity of Lemma \ref{sep} in the case of a bounded
domains with Lipschitz boundary as well.
\end{remark}

We now show the separation property. The restriction $\alpha >0$ allows us
to apply a comparison argument. Unfortunately, these bounds are not uniform
as $\alpha \rightarrow 0^{+}$.

\begin{lemma}
\label{sep}Let the assumptions of Theorem \ref{expo_2} be satisfied and let $%
\alpha >0$. Let $\left\Vert \varphi _{0}\right\Vert _{L^{\infty }\left(
\Omega \right) }\leq 1-\overline{\delta }$, for some $\overline{\delta }>0$.
Then, the solution $\varphi \left( t\right)=S(t)\varphi_0 $ is
instantaneously bounded, i.e., for every $\tau >0,$ we have
\begin{equation}
\sup_{t\geq \tau }\left\Vert \varphi \left( t\right) \right\Vert _{L^{\infty
}\left( \Omega \right) }\leq 1-\delta _{\alpha ,\tau ,\overline{\delta }%
}\left( \left\Vert \varphi _{0}\right\Vert _{\mathcal{Y}_{m}}\right) ,
\label{sing_h1}
\end{equation}%
where the constant $\delta _{\alpha ,\tau ,\overline{\delta }}>0$ depends on
$\alpha ^{-1},$ $\tau ,$ $\overline{\delta }$ and the initial data $\varphi
_{0}$ in $\mathcal{Y}_{m}.$ Moreover, there exists a time $t_{0}=t_{0}\left(
\left\Vert \varphi _{0}\right\Vert _{\mathcal{Y}_{m}}\right) >0$, depending
on the initial data, and there are constants $C_{\alpha }^{\prime},\delta
_{\alpha }>0,$ independent of the initial data, such that
\begin{equation}
\sup_{t\geq t_{0}}\left\Vert \varphi \left( t\right) \right\Vert _{L^{\infty
}\left( \Omega \right) }\leq 1-\delta _{\alpha }.  \label{sing_h2}
\end{equation}%
In particular, the separation property $\left\Vert F^\prime\left( \varphi
\left( t\right) \right) \right\Vert _{L^{\infty }\left( \Omega \right) }\leq
C_{\alpha }^{\prime}$ holds for all $t\geq t_{0}.$
\end{lemma}

\begin{proof}
\emph{Step 1.} To prove the instantaneous boundedness (\ref{sing_h1}), we
rewrite equation (\ref{chem}) as a first-order ordinary differential
equation:%
\begin{equation}
\alpha \partial _{t}\varphi +F^{\prime }\left( \varphi \right) +a\left(
x\right) \varphi =\mu +J\ast \varphi =:h_{\mu ,\varphi }.  \label{ODE}
\end{equation}%
Recall that (\ref{ODE}) is also subject to the initial condition%
\begin{equation}
\varphi \left( 0\right) =\varphi _{0},\;\text{ with }\left\vert \varphi
_{0}\right\vert <1,\;\text{ a.e. in }\Omega ,  \label{ODEic}
\end{equation}%
and that we have (cf. Theorem \ref{well_sing})
\begin{equation}
\left\vert \varphi \left( t\right) \right\vert <1, \quad \text{ a.e. in }
Q_{+}= \mathbb{R}_{+}\times \Omega .  \label{lp*}
\end{equation}

Next, according to estimate (\ref{chemp}) and using the embedding $%
H^{2}\left( \Omega \right) \subset L^{\infty }\left( \Omega \right) $, we
have%
\begin{equation}
\sup_{t\geq \tau }\left\Vert \mu \left( t\right) \right\Vert _{L^{\infty
}\left( \Omega \right) }\leq \mathcal{C=C}_{\tau ,\alpha },
\label{chem_bound}
\end{equation}%
with an appropriate positive constant $\mathcal{C}_{\tau ,\alpha }.$
Moreover using (\ref{lp*}) we readily obtain%
\begin{equation}
\sup_{t\geq \tau }\left\Vert \left( J\ast \varphi \right) \left( t\right)
\right\Vert _{L^{\infty }\left( \Omega \right) }\leq C_{\alpha ,\tau ,J},
\label{convlinf}
\end{equation}%
for every $\tau >0,$ which in light of (\ref{chem_bound}) and (\ref{convlinf}%
), yields%
\begin{equation*}
\sup_{t\geq \tau }\left\Vert h_{\mu ,\varphi }\left( t\right) \right\Vert
_{L^{\infty }\left( \Omega \right) }\leq C_{\alpha ,\tau }.
\end{equation*}%
Therefore, on account of assumptions (H10)-(H11), bound (\ref{sing_h1})
follows from the application of the comparison principle (see, e.g, \cite[%
Proposition A.3]{MZ}) to (\ref{ODE})-(\ref{ODEic}) on $\left[ \tau ,\infty
\right) $.

\emph{Step 2. }In order to deduce the uniform estimate (\ref{sing_h2}) we
shall first derive the following dissipative estimate:%
\begin{equation}  \label{dis9}
\mathcal{E}\left( \varphi \left( t\right) \right) +\int_{t}^{t+1}\left(
\left\Vert \nabla \mu \left( s\right) \right\Vert _{H}^{2}+\alpha \left\Vert
\partial _{t}\varphi \left( s\right) \right\Vert _{H}^{2}\right) ds \leq
\mathcal{E}\left( \varphi _{0}\right) e^{-\kappa t}+C_{m},
\end{equation}%
for all $t\geq 0$, for some positive constant $C_{m}$ independent of the
initial data, time and $\alpha \geq 0$, but which depends on $m\in \left(
0,1\right) $ such that $\left\vert \left\langle \varphi _{0}\right\rangle
\right\vert \leq m.$ The proof of (\ref{dis9}) follows the same lines of
\cite[Proposition 2]{FG1} and \cite[Corollary 2]{CFG}. We briefly mention
some details. Let us thus multiply equation $\mu =a\varphi -J\ast \varphi
+F^{\prime }(\varphi )+\alpha \partial _{t}\varphi $ by $\varphi $ in $%
L^{2}(\Omega )$ and integrate over $\Omega $. We obtain
\begin{equation}
\left\langle \mu , \varphi \right\rangle = \frac{1}{4} \int_{\Omega
}\int_{\Omega }J(x-y)(\varphi (x)-\varphi (y))^{2}dxdy+\left\langle
F^{\prime }(\varphi),\varphi \right\rangle +\frac{\alpha }{2}\frac{d}{dt}%
\left\Vert \varphi \right\Vert _{H}^{2}.  \label{muphi}
\end{equation}%
Observe now that, due to the singular character of $F^{\prime }$, we can
find $C_{F}>0$ such that
\begin{equation}
F^{\prime }(s) s\geq F( s) -C_{F},\quad \forall s\in (-1,1),  \label{lbbb}
\end{equation}%
Then, using (\ref{lbbb}), we obtain
\begin{eqnarray}
&&\langle\mu,\varphi\rangle\geq \frac{1}{4} \int_{\Omega }\int_{\Omega
}J(x-y)(\varphi (x)-\varphi (y))^{2}dxdy+\int_{\Omega }F(\varphi (t))dx
\label{muphi2} \\
&&-C_{F}|\Omega |+\frac{\alpha }{2}\frac{d}{dt}\left\Vert \varphi
\right\Vert _{H}^{2}.  \notag
\end{eqnarray}%
We also have (note that $\left\langle \partial _{t}\varphi \right\rangle =0$%
)
\begin{equation*}
\langle\mu ,\varphi \rangle=\left\langle \mu -\left\langle \mu \right\rangle
,\varphi \right\rangle +\left\langle \mu \right\rangle \left\vert \Omega
\right\vert \left\langle \varphi \right\rangle \leq c_{\Omega }\Vert \nabla
\mu \Vert _{H}\Vert \varphi \Vert _{H}+c_{m},
\end{equation*}%
and then, by means of (\ref{lbbb}), from \eqref{muphi2} we have
\begin{align*}
& \frac{1}{8}\int_{\Omega }\int_{\Omega }J(x-y)(\varphi (x)-\varphi
(y))^{2}dxdy+\frac{1}{2}\int_{\Omega }F(\varphi )dx \\
& +\frac{c}{2}\int_{\Omega }|\varphi |^{2+2q}dx-c+\frac{\alpha }{4}\frac{d}{%
dt}\left\Vert \varphi \right\Vert _{H}^{2} \\
& \leq \Vert \nabla \mu \Vert _{H}^{2}+\frac{c_{\Omega }^{2}}{2}\Vert
\varphi \Vert _{H}^{2}+c_{m},
\end{align*}%
for appropriate constants $c_{m},c>0$, independent of the initial data, time
and $\alpha $. Therefore, we deduce
\begin{align}
& \frac{1}{8}\int_{\Omega }\int_{\Omega }J(x-y)(\varphi (x)-\varphi
(y))^{2}dxdy+\frac{1}{2}\int_{\Omega }F(\varphi )dx+\frac{\alpha }{4}\frac{d%
}{dt}\left\Vert \varphi \right\Vert _{H}^{2}  \label{hr} \\
& \leq \Vert \nabla \mu \Vert _{H}^{2}+c_{m}  \notag
\end{align}%
and, hence, by virtue of (\ref{idesing}) and (\ref{hr}), we get
\begin{equation}
\frac{d}{dt}\left( \mathcal{E}(\varphi )+\alpha \left\Vert \varphi
\right\Vert _{H}^{2}\right) +c\mathcal{E}(\varphi )+\Vert \nabla \mu \Vert
_{H}^{2}+\alpha \left\Vert \partial _{t}\varphi \right\Vert _{H}^{2}\leq
c_{m},  \label{dis9bis}
\end{equation}%
for all $t\geq 0$. By means of Gronwall's lemma we thus easily infer (\ref%
{dis9}) from (\ref{dis9bis}). From (\ref{dis9}), we can now find a time $%
t_{\#}=t_{\#}\left( \mathcal{E}(\varphi _{0})\right) >0$ such that%
\begin{equation}
\sup_{t\geq t_{\#}}\left[ \mathcal{E}(\varphi \left( t\right)
)+\int_{t}^{t+1}\Vert \nabla \mu \left( s\right) \Vert _{H}^{2}+\alpha
\left\Vert \partial _{t}\varphi \left( s\right) \right\Vert _{H}^{2}ds\right]
\leq R_{m}^{\#},  \label{babs}
\end{equation}%
for some $R_{m}^{\#}>0$, independent of $t$, $\alpha $ and the initial data.
With estimate (\ref{babs}) at hand, we can now argue as in the proof of
Lemma \ref{chem_lemma} to get the bound:
\begin{equation}
\sup_{t\geq t_{\#}+1}\left( \left\Vert \mu \left( t\right) \right\Vert
_{H^{2}\left( \Omega \right) }+\alpha \left\Vert \partial _{t}\varphi \left(
t\right) \right\Vert _{H}\right) \leq R_{\alpha ,m},  \label{babs2}
\end{equation}%
for some positive constant $R_{\alpha ,m}$ which depends on $R_{m}^{\#}$ and
$\alpha ^{-1}>0$ only. Finally using (\ref{babs2}) and then arguing as in
\emph{Step 1} above, we can easily arrive at the following inequality:%
\begin{equation}
\sup_{t\geq t_{0}}\left\Vert h_{\mu ,\varphi }\left( t\right) \right\Vert
_{L^{\infty }\left( \Omega \right) }\leq \mathcal{R}_{\alpha .m},
\label{babs3}
\end{equation}%
for some positive constant $\mathcal{R}_{\alpha ,m}$ which only depends on $%
R_{\alpha ,m}$, $q$, and the other fixed parameters of the problem. Here $%
t_{0}$ depends on $t_{\#}$. Thus, on account of (\ref{babs3}), we can once
again apply the comparison principle (see, e.g., \cite[Corollary A.1]{MZ})
to (\ref{ODE})-(\ref{ODEic}), to deduce the existence of a positive constant
$\delta _{1}=\delta _{1}\left( \mathcal{R}_{\alpha ,m}\right) $ which is
independent of $\varphi _{0}$ and time, such that $\left\vert \varphi \left(
t\right) \right\vert \leq 1-\delta _{1}$, a.e. in $\Omega ,$ for all $t\geq
t_{0}$. Inequality (\ref{sing_h2}) is now proven.
\end{proof}

In what follows, we derive as in Section 2.2 some basic properties of $%
S\left( t\right) $ which will be useful to establish the existence of an
exponential attractor. The following proposition, whose proof goes as in
Lemma \ref{dec}, is immediate (see Lemma \ref{unising}).

\begin{proposition}
\label{uniq2} Let the assumptions of Lemma \ref{unising} hold. Then we have
\begin{align}
& \left( \left\Vert \varphi _{1}\left( t\right) -\varphi _{2}\left( t\right)
\right\Vert _{V^{\prime }}^{2}+\alpha \left\Vert \varphi _{1}\left( t\right)
-\varphi _{2}\left( t\right) \right\Vert _{H}^{2}+C\left\vert
M_{1}-M_{2}\right\vert \right) \\
& \leq \left( \left\Vert \varphi _{1}\left( 0\right) -\varphi _{2}\left(
0\right) \right\Vert _{V^{\prime }}^{2}+\alpha \left\Vert \varphi _{1}\left(
0\right) -\varphi _{2}\left( 0\right) \right\Vert _{H}^{2}+\left\vert
M_{1}-M_{2}\right\vert \right) e^{-\beta t}  \notag \\
& +C\int_{0}^{t}\left( \left\Vert \varphi _{1}\left( s\right) -\varphi
_{2}\left( s\right) \right\Vert _{V^{\prime }}^{2}+\left\vert
M_{1}-M_{2}\right\vert \right) ds,  \notag
\end{align}%
for all $t\geq 0$, where $M_{i}:=\left\langle \varphi _{i}\left( 0\right)
\right\rangle $, for some positive constants $\beta $, $C$ which depend on $%
c_{0}$ and $J$, but are independent of $\alpha .$
\end{proposition}

The following one is also straightforward.

\begin{proposition}
\label{uniq3} Let the assumptions of Lemma \ref{unising} be satisfied. Then,
for every $\tau >0$, the following estimate holds:%
\begin{align}
& \left\Vert \partial _{t}\varphi _{1}-\partial _{t}\varphi _{2}\right\Vert
_{L^{2}(\left[ \tau ,t\right] ;D\left( A_{N}\right) ^{^{\prime
}})}^{2}+c_{0}\int_{0}^{t}\left\Vert \varphi _{1}\left( s\right) -\varphi
_{2}\left( s\right) \right\Vert _{H}^{2}ds  \label{comp1bis} \\
& \leq C_{m,\tau }e^{\kappa t}\left( \left\Vert \varphi _{1}\left( 0\right)
-\varphi _{2}\left( 0\right) \right\Vert _{V^{\prime }}^{2}+\alpha
\left\Vert \varphi _{1}\left( 0\right) -\varphi _{2}\left( 0\right)
\right\Vert _{H}^{2}+\left\vert M_{1}-M_{2}\right\vert \right) ,  \notag
\end{align}%
for all $t\geq \tau $, where $C_{m,\tau }$ and $\kappa >0$ also depend on $%
c_{0},$ $\Omega ,$ $\alpha >0$ and $J.$
\end{proposition}

\begin{proof}
In light of the separation property, the proof goes essentially as the one
of Lemma \ref{lipdif}.
\end{proof}

The next lemma gives the uniform H\"{o}lder continuity of $t \mapsto
S(t)\varphi_0$ with respect to the $H$-norm.

\begin{lemma}
\label{hctimebis}Let the assumptions of Theorem \ref{well_sing} be
satisfied. Consider $\varphi \left( t\right) =S\left( t\right) \varphi _{0}$
with $\varphi _{0}\in \mathcal{Y}_{m}$. Then, for every $\tau >0$, there
holds:
\begin{equation}
\left\Vert \varphi \left( t\right) -\varphi \left( s\right) \right\Vert
_{H}\leq C_{m,\alpha ,\tau }\left\vert t-s\right\vert ,\text{ }\forall
t,s\geq \tau ,  \label{hclinf2bis}
\end{equation}%
where the positive constant $C_{m,\alpha ,\tau }$ is independent of initial
data, $\varphi $ and $t,s$.
\end{lemma}

\begin{proof}
According to (\ref{chemp}), we have the bound
\begin{equation}
\sup_{t\geq \tau }\left\Vert \partial _{t}\varphi \left( t\right)
\right\Vert _{H}\leq C_{m,\alpha ,\tau }.  \label{prec}
\end{equation}%
Observing that%
\begin{equation*}
\varphi \left( t\right) -\varphi \left( s\right) =\int_{s}^{t}\partial
_{t}\varphi \left( z\right) dz,
\end{equation*}%
we readily deduce (\ref{hclinf2bis}), thanks to (\ref{prec}).
\end{proof}

\noindent \textit{Proof of Theorem \ref{expo_2}}. As in Section 2.2, we
apply the abstract result of Proposition \ref{abstract}. In light of the
separation property in Lemma \ref{sep}, it is not difficult to realize that
there exists an absorbing set of the following form%
\begin{equation*}
\mathbb{B}\left( \delta _{\alpha },m\right) :=\left\{ \varphi \in \mathcal{Y}%
_{m}\,:\,-1+\delta _{\alpha }\leq \varphi \leq 1-\delta _{\alpha },\;\text{
a.e. in }\Omega \right\} ,
\end{equation*}%
for a suitable constant $\delta _{\alpha }$. We endow $\mathbb{B}\left(
\delta _{\alpha },m\right) $ with the metric of $\mathcal{H}=H$, and
reasoning as above (see Section 2.2), we can suppose that $\mathbb{B}\left(
\delta _{\alpha },m\right) $ is semi-invariant for $S\left( t\right) $ for $%
t\geq 0$. On the other hand, due to the results proven in this section, we
have%
\begin{equation*}
\sup_{t\geq 0}\left( \left\Vert F^{\prime }(\varphi \left( t\right)
)\right\Vert _{L^{\infty }\left( \Omega \right) }+\left\Vert \mu \left(
t\right) \right\Vert _{H^{2}\left( \Omega \right) }+\left\Vert \partial
_{t}\varphi \left( t\right) \right\Vert _{H}\right) \leq C_{m,\alpha },
\end{equation*}%
for every trajectory $\varphi \left( t\right) $ originating from $\varphi
_{0}\in \mathbb{B}\left( \delta _{\alpha },m\right) $, for some positive
constant $C_{m,\alpha }$ which is independent of $\varphi _{0}\in \mathbb{B}%
\left( \delta _{\alpha },m\right) $. We can now apply Proposition \ref%
{abstract} to the map $\mathbb{S}=S\left( T\right) $ and $\mathcal{H}=H$,
with the same choice of the functional spaces $\mathcal{V}_{1}$, $\mathcal{V}
$ as in (\ref{fsp}), owing to Propositions \ref{uniq2}, \ref{uniq3} and
Lemma \ref{hctimebis}. Consequently, we obtain the (finite-dimensional)
exponential attractor $\mathcal{M}$ for $S\left( t\right) $ restricted to $%
\mathbb{B}\left( \delta _{\alpha },m\right) $ in the $H$-metric. The
attraction property (iii) of Theorem \ref{expo_2} is again a consequence of
the separation property and the basic interpolation inequality%
\begin{equation*}
\left\Vert u\right\Vert _{L^{s}\left( \Omega \right) }\leq C_{s}\left\Vert
u\right\Vert _{H}^{\nu _{s}}\left\Vert u\right\Vert _{L^{\infty }\left(
\Omega \right) }^{1-\nu _{s}},\quad \nu _{s}\in \left( 0,1\right) .
\end{equation*}%
Theorem \ref{expo_2} is thus proved.

\begin{remark}
\label{ACsimilar} In contrast to the results proved in the case of regular
potentials, we cannot show that $\varphi (t)$ is ultimately bounded in $V$%
-norm like in the nonviscous case $\alpha =0$ (cf. Lemma \ref{h1lemma} and %
\eqref{cda_pre}). This can also be understood by formally rewriting the
original equation in the following form
\begin{equation*}
\varphi _{t}=(I-\alpha \Delta )^{-1}\Delta \left( a\varphi -J\ast \varphi
+F^{\prime }(\varphi )\right)
\end{equation*}%
which shows that this equation is much closer to the nonlocal Allen-Cahn
equation (see, for instance, \cite{BC2,BCh,BFRW,Chen,GR} and references
therein). Moreover, there is a close connection between the viscous nonlocal
Cahn-Hilliard equation and the phase-field system investigated in \cite{GS},
namely,
\begin{align}
& \partial _{t}(\varsigma \vartheta +\varphi )=\Delta \vartheta ,
\label{pf2} \\
& \alpha \partial _{t}\varphi -J\ast \varphi +a\varphi +F^{\prime }(\varphi
)=\vartheta ,  \label{pf3}
\end{align}%
in $\Omega \times (0,\infty )$, where $\vartheta $ denotes a rescaled
relative temperature and $\varsigma >0$. Indeed, if we let $\varsigma $ go
to $0$ formally, then we get
\begin{equation*}
\partial _{t}\varphi =\Delta \vartheta .
\end{equation*}%
Thus we obtain
\begin{equation}
\partial _{t}\varphi =\Delta \left( \alpha \partial _{t}\varphi -J\ast
\varphi +a\varphi +F^{\prime }(\varphi )\right) ,  \label{vnlCH}
\end{equation}%
that is, the viscous nonlocal Cahn-Hilliard equation. It would be
interesting to investigate the connections between the phase-field system (%
\ref{pf2})-(\ref{pf3}) and equation \eqref{vnlCH} along the lines of what
was done for the local equations (see, e.g., \cite{GGM} and its references).
\end{remark}

\subsection{Convergence to a single equilibrium}

Also in this case we have all the ingredients to show that each trajectory
does converge to a single equilibrium. We can state the following version of
the {\L }ojasiewicz-Simon theorem whose proof goes exactly, with some minor
modifications, as in Lemma \ref{LS_lemma} (cf. \cite{GG2} also).

\begin{lemma}
\label{LS_lemma2}Let $J$ satisfy (H1) and let $F$ satisfy (H7)-(H11) and be
real analytic on $\left[ -1+\delta ,1-\delta \right] $ . Then, there exist
constants $\theta \in (0,\frac{1}{2}],$ $C>0,$ $\varepsilon >0$ such that
the following inequality holds:%
\begin{equation}
\left\vert \mathcal{E}\left( \varphi \right) -\mathcal{E}\left( \varphi
_{\ast }\right) \right\vert ^{1-\theta }\leq C\left\Vert \mu -\left\langle
\mu \right\rangle \right\Vert _{H},  \label{LSineq2}
\end{equation}%
for all%
\begin{equation*}
\varphi \in \left\{ \psi \in L^{\infty }\left( \Omega \right) \cap \mathcal{Y%
}_{m}:-1+\delta \leq \psi \leq 1-\delta ,\text{ a.e. in }\Omega \right\} ,
\end{equation*}%
provided that $\left\Vert \varphi -\varphi _{\ast }\right\Vert _{H}\leq
\varepsilon .$
\end{lemma}

The analog of Theorem \ref{conv_thm} in the case $\alpha >0$ and singular
potentials $f$ is

\begin{theorem}
\label{conv_thm2} Let the assumptions of Theorem \ref{expo_2} hold. Suppose
in addition that $F$ is real analytic on $\left[ -1+\delta ,1-\delta \right]$%
. Then, any weak solution $\varphi $ to (\ref{nlch})-(\ref{ic}) belonging to
the class (\ref{re4})-(\ref{re6}) satisfies%
\begin{equation*}
\lim_{t\rightarrow \infty }\left\Vert \varphi \left( t\right) -\varphi
_{\ast }\right\Vert _{H}=0,
\end{equation*}%
where $\varphi _{\ast }$ is solution to (\ref{stat}).
\end{theorem}

\begin{proof}
The proof goes essentially along the lines of Theorem \ref{conv_thm2}.
Indeed, it is easier since by virtue of (\ref{LSineq2}) and the energy
identity (\ref{idesing}), one can establish instead of (\ref{convls2tris})
the bound:%
\begin{equation*}
\int_{M}\left\Vert \partial _{t}\varphi \left( s\right) \right\Vert
_{H}ds\leq C_{\alpha },
\end{equation*}%
which entails the integrability of $\partial _{t}\varphi $ in $L^{1}(\tau
,\infty ;H)$. We leave the details to the interested reader (see, also, \cite%
[Section 6]{FIP}).
\end{proof}

\section{Degenerate mobility and logarithmic potential}

\noindent In this section we consider the model proposed in \cite{GZ} (see
also \cite{GL1,GL2}). Thanks to the particular form of the mobility
coefficient, the separation property holds even in absence of viscosity (see
\cite{LP, LP2}). As a consequence, we can prove the existence of an
exponential attractor in this case as well.

Referring to \cite{GZ} for details, we consider the following boundary value
problem
\begin{equation}
\left\{
\begin{array}{ll}
\mu =F^{\prime}\left( \varphi \right) +w\text{,} & \text{in }Q, \\
w\left( x,t\right) =\int_{\Omega }J\left( x-y\right) \left( 1-2\varphi\left(
y,t\right) \right) dy\text{,} & \left( x,t\right) \in Q, \\
\partial _{t}\varphi =\nabla \cdot \left( \kappa \left(\cdot, \varphi
\right) \nabla \mu\right) \text{,} & \text{in }Q\text{,} \\
\kappa \left( \varphi \right) \partial _{\mathbf{n}}\mu =0\text{,} & \text{%
on }\Gamma \times \left( 0,T\right) ,%
\end{array}%
\right.  \label{degm}
\end{equation}%
subject to the initial condition%
\begin{equation}
\varphi _{\mid t=0}=\varphi _{0}, \quad \text{ in }\Omega .  \label{degmic}
\end{equation}%
On account of \cite[Section 2]{GZ} we assume the following hypotheses:

\begin{enumerate}
\item[(H12)] $F\left( \varphi \right) =\varphi \log \varphi +\left(
1-\varphi \right) \log \left( 1-\varphi \right) .$

\item[(H13)] The mobility $\kappa $ has the form%
\begin{equation}
\kappa \left(x, \varphi \right) =\frac{b\left( x,\left\vert \nabla \varphi
\right\vert \right) }{F^{\prime \prime}\left( \varphi \right) },  \label{mob}
\end{equation}%
where the Carath\'{e}odory function \thinspace $b\left( x,\left\vert
s\right\vert \right) :\Omega \times \mathbb{R}_{+}\rightarrow \mathbb{R}_{+}$
satisfies:%
\begin{equation}
\left( b\left( x,\left\vert s_{1}\right\vert \right) s_{1}-b\left(
x,\left\vert s_{2}\right\vert \right) s_{2}\right) \left( s_{1}-s_{2}\right)
\geq \alpha _{0}\left\vert s_{1}-s_{2}\right\vert ^{2},  \label{a1}
\end{equation}%
for all $s_{1},s_{2}\in \mathbb{R}^{d}$, a.e. in $\Omega ,$ for some $\alpha
_{0}>0,$
\begin{equation}
\left\vert b\left( x,\left\vert s_{1}\right\vert \right) s_{1}-b\left(
x,\left\vert s_{2}\right\vert \right) s_{2}\right\vert \leq \alpha
_{1}\left\vert s_{1}-s_{2}\right\vert ,  \label{a2}
\end{equation}%
for all $s_{1},s_{2}\in \mathbb{R}^{d}$, a.e. in $\Omega ,$ for some $\alpha
_{1}>0.$
\end{enumerate}

The notion of weak solution to problem (\ref{degm})-(\ref{degmic}) is given
by

\begin{definition}
\label{defdeg}Let $Q:=\Omega \times \left( 0,T\right) $. A function $\varphi
$ is called a solution of (\ref{degm})-(\ref{degmic}) if%
\begin{equation}
\varphi \in L^{\infty }\left( Q\right) \cap L^{2}\left( \left[ \delta ,T%
\right] ;V\right)  \label{reg1}
\end{equation}%
with
\begin{equation}
\partial _{t}\varphi \in L^{2}(\left[ \delta ,T\right] ;V^{\prime }),\quad
w\in L^{\infty }\left( \left[ \delta ,T\right] ;W^{1,\infty }\left( \Omega
\right) \right)  \label{reg2}
\end{equation}%
satisfy (\ref{degm}) for every $\delta >0$, the chemical potential obeys%
\begin{equation*}
\int_{0}^{T}\int_{\Omega }\kappa \left( x,\varphi \right) \left\vert \nabla
\mu \right\vert ^{2}dxds<\infty ,
\end{equation*}%
and the following identity holds
\begin{equation}
\left\langle \partial _{t}\varphi ,\psi \right\rangle +\left( \kappa \left(
\cdot ,\varphi \right) \nabla \mu ,\nabla \psi \right) =0,\quad \forall \psi
\in V,\;\text{ a.e. on }\left( 0,T\right) .  \label{degwf}
\end{equation}
\end{definition}

The following well-posedness result was proven in \cite[Theorem 3.5 and
Corollary 3.7]{GZ} (see also \cite{LP, LP2}).

\begin{theorem}
\label{degwell} Let assumptions (H1) and (H13)-(H14) be satisfied. Consider
\begin{align*}
\mathcal{Y}_{m_{1},m_{2}}:=\{\psi \in L^{\infty }(\Omega )\,:\,& 0\leq \psi
\leq 1,\text{ a.e. in }\Omega ,\;F(\psi )\in L^{1}(\Omega ), \\
& 0<m_{1}\leq \left\langle \psi \right\rangle \leq m_{2}<1\}
\end{align*}%
where $m_{1}$, $m_{2}$ are fixed and the metric is given by (\ref{metric}).
If $\varphi _{0}\in \mathcal{Y}_{m_{1},m_{2}}$, then there exists a unique
solution to problem (\ref{degm})-(\ref{degmic}) in the sense of Definition %
\ref{defdeg} such that $\left\langle \varphi \left( t\right) \right\rangle
=\left\langle \varphi _{0}\right\rangle $, for all $t\geq 0$. Moreover, we
have
\begin{equation}
F^{\prime }\left( \varphi \right) \in L^{\infty }\left( \left[ \delta ,T%
\right] ;L^{\infty }\left( \Omega \right) \right) ,\quad \mu \in L^{\infty
}\left( \left[ \delta ,T\right] ;L^{\infty }\left( \Omega \right) \right)
\cap L^{2}\left( \left[ \delta ,T\right] ;V\right) ,  \label{reg3}
\end{equation}%
for all $\delta >0.$
\end{theorem}

\begin{remark}
We can take $\delta =0$ in (\ref{reg1}), (\ref{reg2}) and (\ref{reg3})
provided that $\varphi _{0}\in \mathcal{Y}_{m_{1},m_{2}}$ and, \emph{more
importantly}, $F^{\prime }\left( \varphi _{0}\right) \in L^{\infty }\left(
\Omega \right) $ (see \cite[Corollary 3.7]{GZ}). Furthermore, by \cite[Lemma
2.1]{LP} the solution $\varphi \left( t\right) $ also belongs to $C\left( %
\left[ 0,T\right] ;H\right) $ (see Proposition \ref{contdep} below)$.$ Thus,
in view of (\ref{reg3}), $\varphi \in C\left( \left[ 0,T\right] ;L^{p}\left(
\Omega \right) \right) $ for every $p\in \lbrack 2,\infty )$, using
interpolation and the $L^{\infty }$-bound of $\varphi .$
\end{remark}

The following separation property for the (weak) solutions given by Theorem %
\ref{degwell} was proven in \cite[Theorem 2.1]{LP} (see also \cite{LP2}).

\begin{theorem}
\label{sep_thm} Let assumptions of Theorem \ref{degwell} and let $\varphi $
be a weak solution. Then there exist positive constants $T_{0},C,\delta $
(with $\delta ,C$ indepedent of the initial data) such that%
\begin{equation}
\delta \leq \varphi \left( x,t\right) \leq 1-\delta ,\quad \text{ a.e. on }%
\Omega \times \left( T_{0},\infty \right)  \label{sep_deg}
\end{equation}%
and%
\begin{equation}
\sup_{t\geq T_{0}}\left\Vert \mu \left( t\right) \right\Vert _{L^{\infty
}\left( \Omega \right) }\leq C.  \label{chem_b}
\end{equation}
\end{theorem}

\begin{remark}
\label{r1}The proof of Theorem \ref{sep_thm} is given in \cite[Lemma
3.1-Lemma 3.3]{LP} (see also \cite{LP2}, for a simplifying argument),
assuming that $b\left( x,\left\vert s\right\vert \right) \equiv const,$ by
exploiting an Alikakos-Moser iteration scheme for suitable powers of the
functions $\log \left( 1-\varphi \right) $ and $\log \left( \varphi \right) .
$ The special form of $\kappa $ (see \eqref{mob}) plays an \emph{essential}
role in the calculations. The proofs in \cite{LP, LP2} can be easily adapted
without too much difficulty to the case of nonconstant functions $b\left(
x,\left\vert s\right\vert \right) $. Indeed, from assumptions (\ref{a1})-(%
\ref{a2}) there holds $0<\alpha _{0}\leq b\left( x,\left\vert s\right\vert
\right) \leq \alpha _{1}$, for almost any $\left( x,s\right) \in \Omega
\times \mathbb{R}^{d}$. The same bound (\ref{sep_deg}) can be deduced for
more general functions, that is,
\begin{equation*}
f\in C^{2}\left( 0,1\right) :f\text{ strictly convex, Image}(f^{^{\prime
}})^{-1}=\left[ 0,1\right] ,\text{ }\frac{1}{f^{^{\prime \prime }}}\text{
strictly concave,}
\end{equation*}%
see \cite{LP2}.
\end{remark}

\begin{remark}
\label{r2}Note that the separation property (\ref{sep_deg}) implies that the
problem (\ref{degm})-(\ref{degmic}) is non-degenerate for all $t\geq T_{0}$.
Indeed, we have
\begin{equation}
\alpha _{0}\delta \left( 1-\delta \right) \leq \kappa \left( \cdot ,\varphi
\right) \leq \alpha _{1},\quad \text{ a.e. in }\Omega \times \left(
T_{0},\infty \right) .  \label{nondm}
\end{equation}%
It is also worth mentioning that (\ref{sep_deg}) and (\ref{nondm}) hold
almost everywhere in $\Omega \times \left( \tau ,\infty \right) ,$ uniformly
with respect to bounded sets of initial data in $\mathcal{Y}_{m_{1},m_{2}}$.
More precisely, for every ball of radius $R$, there exists $T_{0}=T_{0}(R)>0$
such that \eqref{sep_deg} and \eqref{chem_b} hold (see \cite{LP}, \cite{LP2}%
).
\end{remark}

\begin{theorem}
Let the assumptions of Theorem \ref{degwell} be satisfied. Then there holds%
\begin{equation*}
\sup_{t\geq T_{0}}\left\Vert \varphi \right\Vert _{C^{\alpha /2,\alpha
}\left( \left[ t,t+1\right] \times \overline{\Omega }\right) }\leq C,
\end{equation*}%
for some $\alpha \in \left( 0,1\right) .$
\end{theorem}

\begin{proof}
As in the proof of Lemma \ref{holderb}, we can rewrite (\ref{degm}) as (\ref%
{sysnew1}) on $\,t\in \left( T_{0},\infty \right) ,$ for the function%
\begin{equation*}
a\left( x,\varphi ,\nabla \varphi \right) :=b\left( x,\left\vert \nabla
\varphi \right\vert \right) \nabla \varphi +\kappa \left( x,\varphi \right)
\nabla w.
\end{equation*}%
Notice that, from Remarks \ref{r1} and \ref{r2}, we have $a\left( x,\varphi
,\nabla \varphi \right) \nabla \varphi \geq \frac{\alpha _{0}}{2}\left\vert
\nabla \varphi \right\vert ^{2}-C_{\delta }$ and $\left\vert a\left(
x,\varphi ,\nabla \varphi \right) \right\vert \leq \alpha _{1}\left\vert
\nabla \varphi \right\vert +C_{\delta };$ hence \cite[Corollary 4.2]{Du}
still applies and this entails the desired estimate.
\end{proof}

The main result of this section is contained in the following

\begin{theorem}
\label{expo_log}Let the assumptions of Theorem \ref{degwell} be satisfied.
There exists an exponential attractor $\mathcal{M}$ bounded in $C^{\alpha
}\left( \overline{\Omega }\right) ,\alpha \in \left( 0,1\right) $ and
compact in $H,$ for the dynamical system $\left( \mathcal{Y}%
_{m_{1},m_{2}},S\left( t\right) \right) $ associated with (\ref{degm})-(\ref%
{degmic}), satisfying the following properties:

(i) Semi-invariance: $S\left( t\right) \mathcal{M}\subset \mathcal{M}$, for
every $t\geq 0.$

(ii) The separation property (\ref{sep_deg}) holds for every $\varphi \in
\mathcal{M}$.

(iii) Exponential attraction:%
\begin{equation*}
dist_{L^{s}\left( \Omega \right) }\left( S\left( t\right) \mathcal{Y}%
_{m_1,m_2},\mathcal{M}\right) \leq Ce^{-\lambda t},
\end{equation*}%
for some positive constants $C$, $\lambda $ and any $s\in \left[ 2,\infty
\right) .$

(iv) Finite dimensionality:%
\begin{equation*}
\dim _{F}(\mathcal{M},V^{\prime})\leq C<\infty .
\end{equation*}
\end{theorem}

Consequently, we also have the following

\begin{corollary}
Let the assumptions of Theorem \ref{expo_log} be satisfied. The problem (\ref%
{degm})-(\ref{degmic}) possesses a finite dimensional global attractor $%
\mathcal{A}$, $\dim _{F}(\mathcal{A},V^{\prime})<\infty $.
\end{corollary}

In what follows, we derive as in Sections 2 and 3 some basic properties of $%
S\left( t\right) $ which will be useful in order to establish the existence
of an exponential attractor. The following proposition shows that the
semigroup $S\left( t\right) $ is Lipschitz continuous in the $H$-norm with
respect to the initial data.

\begin{proposition}
\label{contdep}Let $\varphi _{i},$ $i=1,2$, be a pair of weak solutions
corresponding to a pair of initial data $\varphi _{i}\left( 0\right) $,
satisfying the assumptions of Theorem \ref{degwell}. Then there holds%
\begin{equation}
\left\Vert \varphi _{1}\left( t\right) -\varphi _{2}\left( t\right)
\right\Vert _{H}^{2}+\int_{0}^{t}\left\Vert \varphi _{1}\left( s\right)
-\varphi _{2}\left( s\right) \right\Vert _{V}^{2}ds\leq \lambda e^{\lambda
t}\left\Vert \varphi _{1}\left( 0\right) -\varphi _{2}\left( 0\right)
\right\Vert _{H}^{2},  \label{Lip_deg}
\end{equation}%
for all $t\geq 0$, for some positive constant $\lambda $ independent of $t.$
\end{proposition}

\begin{proof}
According to Definition \ref{defdeg}, we have that $\varphi :=\varphi
_{1}-\varphi _{2}$ satisfies the following weak formulation%
\begin{equation}
\left\langle \partial _{t}\varphi ,\psi \right\rangle +\left( \kappa \left(
\cdot ,\varphi _{1}\right) \nabla \mu _{1}-\kappa \left( \cdot ,\varphi
_{2}\right) \nabla \mu _{2},\nabla \psi \right) =0\text{, }\forall \psi \in V%
\text{, a.e. on }\left( 0,T\right) ,  \label{wfdiff}
\end{equation}%
where $\mu _{i}=F^{\prime }\left( \varphi _{i}\right) +w_{i}$. Testing
equation (\ref{wfdiff}) with $\psi =\varphi _{1}-\varphi _{2},$ we deduce%
\begin{align}
& \frac{1}{2}\frac{d}{dt}\left\Vert \varphi _{1}-\varphi _{2}\right\Vert
_{H}^{2} +\left( b\left( \cdot ,\left\vert \nabla \varphi _{1}\right\vert
\right) \nabla \varphi _{1}-b\left( \cdot ,\left\vert \nabla \varphi
_{2}\right\vert \right) \nabla \varphi _{2},\nabla \varphi _{1}-\nabla
\varphi _{2}\right)  \label{indeg1} \\
& =-\left( b\left( \cdot ,\left\vert \nabla \varphi _{1}\right\vert \right)
\varphi _{1}\left( 1-\varphi _{1}\right) \nabla w_{1}-b\left( \cdot
,\left\vert \nabla \varphi _{2}\right\vert \right) \varphi _{2}\left(
1-\varphi _{2}\right) \nabla w_{2},\nabla \varphi _{1}-\nabla \varphi
_{2}\right) ,  \notag
\end{align}%
for all $t\geq 0$. Recall that, on account of Theorem \ref{degwell}, $%
\varphi _{i}$ are bounded. Therefore, thanks to (H1), we have the following
a priori bounds:
\begin{equation}
\sup_{t\geq 0}\left\Vert \nabla w_{i}\left( t\right) \right\Vert _{L^{\infty
}\left( \Omega \right) }\leq C,\text{ }i=1,2,  \label{bdeg1}
\end{equation}%
for a suitable positive constant \thinspace $C$. Thus, in view of
assumptions (\ref{a1}) and (\ref{a2}) (see, also, Remark \ref{r1}), and
exploiting (\ref{bdeg1}), we obtain%
\begin{align}
& \frac{1}{2}\frac{d}{dt}\left\Vert \varphi _{1}\left( t\right) -\varphi
_{2}\left( t\right) \right\Vert _{H}^{2}+\alpha _{0}\left\Vert \nabla \left(
\varphi _{1}\left( t\right) -\varphi _{2}\left( t\right) \right) \right\Vert
_{H}^{2}  \label{indeg2} \\
& \leq C_{\alpha _{0}}\left( \left\Vert \nabla w_{1}\right\Vert _{L^{\infty
}\left( \Omega \right) }^{2}+\left\Vert \nabla w_{2}\right\Vert _{L^{\infty
}\left( \Omega \right) }^{2}\right) \left\Vert \varphi _{1}\left( t\right)
-\varphi _{2}\left( t\right) \right\Vert _{H}^{2}  \notag \\
& +\frac{\alpha _{0}}{2}\left\Vert \nabla \left( \varphi _{1}\left( t\right)
-\varphi _{2}\left( t\right) \right) \right\Vert _{H}^{2},  \notag
\end{align}%
which easily yields (\ref{Lip_deg}) on account of Gronwall's inequality.
\end{proof}

As already mentioned, the crucial step in order to establish the existence
of an exponential attractor is the validity of so-called smoothing property
for the difference of two solutions. This is given by

\begin{lemma}
\label{dec_deg} Let the assumptions of Proposition \ref{contdep} hold. Then,
for every $t\geq T_{0}$ the following estimates hold:%
\begin{align}
&\left\Vert \varphi _{1}\left( t\right) -\varphi _{2}\left( t\right)
\right\Vert _{H}^{2} \leq e^{-\lambda _{0}t}\left\Vert \varphi _{1}\left(
0\right) -\varphi _{2}\left( 0\right) \right\Vert _{H}^{2}  \label{diff1} \\
&+C_{\delta }\int_{0}^{t}\left\Vert \varphi _{1}\left( s\right) -\varphi
_{2}\left( s\right) \right\Vert _{H}^{2}ds,  \notag \\
&\left\Vert \partial _{t}\varphi _{1}-\partial _{t}\varphi _{2}\right\Vert
_{L^{2}(\left[ T_{0},t\right] ;V^{^{\prime }})}^{2}+\int_{0}^{t}\left\Vert
\varphi _{1}\left( s\right) -\varphi _{2}\left( s\right) \right\Vert
_{V}^{2}ds  \label{diff2} \\
&\leq C_{\delta }e^{\lambda t}\left\Vert \varphi _{1}\left( 0\right)
-\varphi _{2}\left( 0\right) \right\Vert _{H}^{2},  \notag
\end{align}%
for some positive constants $C_{\delta },\lambda $,$\lambda _{0}$ which
depend only on $\alpha _{0},$ $\alpha _{1},$ $\delta ,$ $\Omega $ and $J.$
\end{lemma}

\begin{proof}
First, we observe that, due to the inequality (\ref{indeg2}), the separation
property (\ref{sep_deg}) and the fact that $w_{i}\in L^{\infty }\left(
T_{0},\infty ;W^{1,\infty }(\Omega )\right) $ uniformly with respect to
time, it holds%
\begin{equation}
\frac{d}{dt}\left\Vert \varphi _{1}\left( t\right) -\varphi _{2}\left(
t\right) \right\Vert _{H}^{2}+\alpha _{0}\left\Vert \nabla \left( \varphi
_{1}\left( t\right) -\varphi _{2}\left( t\right) \right) \right\Vert
_{H}^{2}\leq C_{\delta }\left\Vert \varphi _{1}\left( t\right) -\varphi
_{2}\left( t\right) \right\Vert _{H}^{2},  \label{diff3}
\end{equation}%
for every $t\geq T_{0}$. Thus, combining (\ref{diff3}) together with Poincar%
\'{e}'s inequality
\begin{equation*}
\left\Vert \varphi \right\Vert _{H}^{2}\leq C_{\Omega }\left( \left\Vert
\nabla \varphi \right\Vert _{H}^{2}+\left\langle \varphi \right\rangle
^{2}\right) ,
\end{equation*}%
and recalling the fact that $\left\langle \partial _{t}\left( \varphi
_{1}-\varphi _{2}\right) \right\rangle =0$, we deduce from (\ref{diff3}),
the following inequality:%
\begin{equation*}
\frac{d}{dt}\left\Vert \varphi _{1}\left( t\right) -\varphi _{2}\left(
t\right) \right\Vert _{H}^{2}+C_{\Omega ,\alpha _{0}}\left\Vert \varphi
_{1}\left( t\right) -\varphi _{2}\left( t\right) \right\Vert _{H}^{2}\leq
C_{\delta }\left\Vert \varphi _{1}\left( t\right) -\varphi _{2}\left(
t\right) \right\Vert _{H}^{2}.
\end{equation*}%
Thus, Gronwall's inequality entails the desired estimate (\ref{diff1}). The
second term on the left-hand side of (\ref{diff2}) is estimated in (\ref%
{Lip_deg}). To estimate the time derivative in (\ref{diff2}), recall that $%
\varphi $ satisfies (\ref{wfdiff}). Thus, arguing as in the proof of
Proposition (\ref{contdep}) (note that the mobility $\kappa \left( \varphi
_{i}\right) $ also satisfies (\ref{nondm})) there holds
\begin{align}
\left\langle \partial _{t}\left( \varphi _{1}\left( t\right) -\varphi
_{2}\left( t\right) \right) ,\psi \right\rangle & =\left( \kappa \left(
\cdot ,\varphi _{1}\right) \nabla \mu _{1}-\kappa \left( \cdot ,\varphi
_{2}\right) \nabla \mu _{2},\nabla \psi \right) \\
& \leq C_{\delta }\left\Vert \nabla \psi \right\Vert _{H}\left\Vert \nabla
\varphi _{1}\left( t\right) -\nabla \varphi _{2}\left( t\right) \right\Vert
_{H},  \notag
\end{align}%
for any test function $\psi \in V$, for all $t\geq T_{0}$. This estimate
together with (\ref{Lip_deg}) gives the desired estimate on the time
derivative in (\ref{diff2}).
\end{proof}

The last ingredient we need is the H\"{o}lder continuity of $t\mapsto
S(t)\varphi _{0}$ in the $V^{\prime}$-norm, namely,

\begin{lemma}
\label{hctime_deg} Let the assumptions of Proposition \ref{contdep} be
satisfied. Consider $\varphi \left( t\right) =S\left( t\right) \varphi _{0}$
with $\varphi _{0}\in \mathcal{Y}_{m_1,m_2}$. Then, there holds
\begin{equation}
\left\Vert \varphi \left( t\right) -\varphi \left( s\right) \right\Vert
_{V^{^{\prime }}}\leq C\left\vert t-s\right\vert ^{1/2},\quad \forall t,s\in
\left[ T_{0},T\right] ,  \label{hclinf2_deg}
\end{equation}%
where the constant $C=C_{\delta ,T,T_{0}}>0$ is independent of initial data,
$\varphi $ and $t,s$.
\end{lemma}

\begin{proof}
Testing equation (\ref{degwf}) with $\mu $, then taking the inner product in
$H$ of $\mu =F^{\prime }\left( \varphi \right) +w$ with $\partial
_{t}\varphi $, and adding the resulting relations we obtain%
\begin{equation*}
\frac{1}{2}\frac{d}{dt}\left( F\left( \varphi \left( t\right) \right)
,1\right) +\int_{\Omega }\kappa \left( x,\varphi \right) \left\vert \nabla
\mu \left( t\right) \right\vert ^{2}dx=-\left( w,\partial _{t}\varphi
\right) ,\text{ }\forall t\geq T_{0}.
\end{equation*}%
By virtue of (\ref{sep_deg}) and (\ref{nondm}), we can integrate this
relation over $\left( t,T\right) $, exploit the basic interpolation
inequality $[V,V^{\prime }]_{2,1/2}=H$, and deduce the following inequality:%
\begin{align}
C_{\delta }\int_{t}^{T}\left\Vert \nabla \mu \left( s\right) \right\Vert
_{H}^{2}ds& \leq C\int_{t}^{T}\left\Vert w\left( s\right) \right\Vert
_{V}\left\Vert \partial _{t}\varphi \left( s\right) \right\Vert _{V^{\prime
}}ds+C_{\delta }  \label{est30} \\
& \leq C\int_{t}^{T}\left\Vert w\left( s\right) \right\Vert _{V}\left\Vert
\nabla \mu \left( s\right) \right\Vert _{H}ds+C_{\delta },  \notag
\end{align}%
for all $t\geq T_{0}$ (here, we have used (\ref{degwf}) again to estimate
the time derivative). Note that (H1) and (\ref{sep_deg}) also give the
estimate
\begin{equation*}
\sup_{t\geq T_{0}}\left\Vert w\left( t\right) \right\Vert _{W^{1,\infty
}\left( \Omega \right) }\leq C_{\delta }.
\end{equation*}%
Consequently, from (\ref{est30}) we deduce that
\begin{equation*}
\sup_{t\geq T_{0}}\int_{t}^{T}\left\Vert \nabla \mu \left( s\right)
\right\Vert _{H}^{2}ds\leq C_{\delta }\left( 1+\left( T-T_{0}\right) \right)
,
\end{equation*}%
which entails%
\begin{equation}
\sup_{t\geq T_{0}}\int_{t}^{T}\left\Vert \partial _{t}\varphi \left(
s\right) \right\Vert _{V^{^{\prime }}}^{2}ds\leq C_{\delta }\left( 1+\left(
T-T_{0}\right) \right) .  \label{hcb3bis}
\end{equation}%
Estimate (\ref{hclinf2_deg}) now follows from (\ref{hcb3bis}).
\end{proof}

\noindent \textit{Proof of Theorem \ref{expo_log}}. We shall essentially
argue as in Section 2.2 by applying Proposition \ref{abstract}. We briefly
mention the details. In light of the separation property in Theorem \ref%
{sep_thm}, it is not difficult to realize that there exists a
(semi-invariant) absorbing set of the following form%
\begin{equation*}
\mathbb{B}_{\delta }:=\left\{ \varphi \in \mathcal{Y}_{0}\cap C^{\alpha
}\left( \overline{\Omega }\right) :\,\delta \leq \varphi \leq 1-\delta ,\;%
\text{ a.e. in }\Omega \right\} .
\end{equation*}%
Therefore, it is sufficient to verify the existence of an exponential
attractor for $S(t)_{\mid \mathbb{B}_{\delta }}.$ Note that due to the above
results, we also have%
\begin{equation*}
\sup_{t\geq 0}\left( \left\Vert \varphi \left( t\right) \right\Vert
_{C^{\alpha }\left( \overline{\Omega }\right) }+\left\Vert \mu \left(
t\right) \right\Vert _{L^{\infty }\left( \Omega \right) }+\left\Vert w\left(
t\right) \right\Vert _{W^{1,\infty }\left( \Omega \right) }\right) \leq
C_{\delta },
\end{equation*}%
for every trajectory $\varphi $ originating from $\varphi _{0}\in \mathbb{B}%
_{\delta }$, for some positive constant $C_{\delta }$ which is independent
of the choice of $\varphi _{0}\in \mathbb{B}_{\delta }$. We can now apply
the abstract result above to the map $\mathbb{S}=S\left( T\right) $ and $%
\mathcal{H}=H$, for a fixed $T\geq T_{0}$ such that $e^{-\lambda _{0}T}<%
\frac{1}{2}$, where $\lambda _{0}>0$ is the same as in Lemma \ref{dec_deg}.
To this end, we introduce the functional spaces%
\begin{equation*}
\mathcal{V}_{1}:=L^{2}\left( \left[ 0,T\right] ;V\right) \cap H^{1}\left( %
\left[ 0,T\right] ;V^{\prime }\right) ,\quad \mathcal{V}:=L^{2}\left( \left[
0,T\right] ;H\right) ,
\end{equation*}%
and note that $\mathcal{V}_{1}$ is compactly embedded into $\mathcal{V}$.
Finally, we introduce the operator $\mathbb{T}:\mathbb{B}_{\delta
}\rightarrow \mathcal{V}_{1}$, by $\mathbb{T}\varphi _{0}:=\varphi \in
\mathcal{V}_{1},$ where $\varphi $ solves (\ref{degm})-(\ref{degmic}) with $%
\varphi \left( 0\right) =\varphi _{0}\in \mathbb{B}_{\delta }$. The maps $%
\mathbb{S}$, $\mathbb{T}$, the spaces $\mathcal{H}$,$\mathcal{V}$,$\mathcal{V%
}_{1}$ thus defined satisfy all the assumptions of Proposition \ref{abstract}
on account of Lemma \ref{dec_deg} (see (\ref{diff1})-(\ref{diff2})).
Therefore, the semigroup $\mathbb{S}(n)=S\left( nT\right) $ generated by the
iterations of the operator $\mathbb{S}:\mathbb{B}_{\delta }\mathbb{%
\rightarrow B}_{\delta }$ possesses a (discrete) exponential attractor $%
\mathcal{M}_{d}$ in $\mathbb{B}_{\delta }$ endowed by the topology of $H$.
In order to construct the exponential attractor $\mathcal{M}$ for the
semigroup $S(t)$ with continuous time, we note that, due to Lemma \ref%
{hctime_deg} and Proposition \ref{contdep}, this semigroup is Lipschitz
continuous with respect to the initial data in the topology of $H$. Besides,
the map $\left( t,\varphi _{0}\right) \mapsto S\left( t\right) \varphi _{0}$
is H\"{o}lder continuous on $\left[ 0,T\right] \times \mathbb{B}_{\delta }$,
where $\mathbb{B}_{\delta }$ is endowed with the metric topology of $%
V^{\prime }$. Hence, the desired exponential attractor $\mathcal{M}$ for the
continuous semigroup $S(t)$ can be obtained by the same standard formula in (%
\ref{st}). Theorem \ref{expo_log} is now proved.

Unfortunately, it does not seem possible to establish the finite
dimensionality of $\mathcal{M}$ in Theorem \ref{expo_log} with respect to
the stronger $H$-metric. This issue is ultimately connected to deriving the
same uniform boundedness of $\varphi(t)$ in $V$-norm (cf. also Remark \ref%
{ACsimilar}). However, in a special case at least, the following
regularizing property holds
\begin{equation}
\sup_{t\geq T_{0}+1}\left( \left\Vert \varphi \left( t\right) \right\Vert
_{V}+\left\Vert \partial _{t}\varphi \right\Vert _{L^{2}\left( \left[ t,t+1%
\right] ;H\right) }\right) \leq C_{\delta },  \label{regh1}
\end{equation}%
provided we assume in addition that
\begin{equation}
b\left( \cdot,\left\vert s\right\vert \right) \equiv b_{0}\left(
\cdot\right) \in L^{\infty }\left( \Omega \right), \quad J\in W^{2,1}(
\mathbb{R}^{d}) .  \label{addit_assum}
\end{equation}%
In this case, the exponential attraction (iii)\ and finite dimensionality
(iv) of $\mathcal{M}$ also holds with respect to the $L^{s}\cap H^{1-\nu}$%
-metric for any $s\geq 2$ and $\nu\in \left( 0,1\right) $, on account of (%
\ref{sep_deg}) and (\ref{regh1}).

Let us briefly explain how to get (\ref{regh1}).

\begin{proposition}
Let the assumptions of Theorem \ref{degwell} hold. In addition, suppose (\ref%
{addit_assum}). Then every weak solution $\varphi $\ of problem (\ref{degm}%
)-(\ref{degmic})\ also satisfies estimate (\ref{regh1}).
\end{proposition}

\begin{proof}
Let $h\left( \varphi\right) :=F^{\prime}\left( \varphi\right) +a,$ where $%
a(x)=(1*J)(x)$. According to (\ref{degwf}), every weak solution $\varphi $
satisfies
\begin{equation}
\left\langle \partial _{t}\varphi ,\psi \right\rangle +\left( \kappa
\left(\cdot, \varphi \right) \nabla h,\nabla \psi \right) =2\left( \kappa
\left(\cdot, \varphi \right) \nabla J\ast \varphi ,\nabla \psi \right),
\label{wf2}
\end{equation}
for all $\psi \in V$ and almost everywhere in $\left( T_{0},\infty \right)$.

Testing this identity with $h\left( \varphi \right) $ yields%
\begin{eqnarray}
&&\frac{d}{dt}\left[ \left( F\left( \varphi \left( t\right) \right)
,1\right) +\left( a,\varphi \left( t\right) \right) \right] +\int_{\Omega
}\kappa \left( x,\varphi \left( t\right) \right) \left\vert \nabla h\left(
\varphi \left( t\right) \right) \right\vert ^{2}dx  \label{4.4} \\
&=&2\left( \kappa \left( \cdot ,\varphi \right) \nabla J\ast \varphi \left(
t\right) ,\nabla h\left( t\right) \right) ,  \notag
\end{eqnarray}%
for all $t\geq T_{0}$. Recalling once again that \eqref{sep_deg} and %
\eqref{chem_b} hold uniformly in $\left( T_{0},\infty \right) $, we can
integrate (\ref{4.4}) over $\left( t,t+1\right) $ to deduce the following
bound:%
\begin{equation}
\int_{t}^{t+1}\left\Vert \nabla h\left( \varphi \left( s\right) \right)
\right\Vert _{H}^{2}ds\leq C_{\delta },\text{ }\forall t\geq T_{0}.
\label{4.5}
\end{equation}%
This gives, on account of the separation property (\ref{sep_deg}) and (H1),
that
\begin{equation}
\sup_{t\geq T_{0}}\int_{t}^{t+1}\left\Vert \nabla \varphi \left( s\right)
\right\Vert _{H}^{2}ds\leq C_{\delta }.  \label{4.6}
\end{equation}%
Finally, testing equation (\ref{degwf}) with $\partial _{t}\varphi $ (this
can be easily justified within an appropriate Galerkin scheme), we find%
\begin{equation}
\left\Vert \partial _{t}\varphi \right\Vert _{H}^{2}+\int_{\Omega }\kappa
\left( x,\varphi \right) F^{\prime \prime }\left( \varphi \right) \nabla
\varphi \cdot \nabla \partial _{t}\varphi dx=\int_{\Omega }\kappa \left(
x,\varphi \right) \nabla w\cdot \nabla \partial _{t}\varphi dx.  \label{4.7}
\end{equation}%
By observing that $\kappa \left( \cdot ,\varphi \right) F^{\prime \prime
}\left( \varphi \right) =b_{0}(\cdot )>0$, by virtue of \eqref{sep_deg} and %
\eqref{chem_b} we can further estimate
\begin{align}
& \left\Vert \partial _{t}\varphi \left( t\right) \right\Vert _{H}^{2}+\frac{%
1}{2}\frac{d}{dt}\int_{\Omega }b_{0}\left( x\right) \left\vert \nabla
\varphi \left( t\right) \right\vert ^{2}dx  \label{4.8} \\
& \leq C_{\delta }\left\Vert \nabla w\left( t\right) \right\Vert
_{V}\left\Vert \nabla \partial _{t}\varphi \left( t\right) \right\Vert
_{V^{^{\prime }}}  \notag \\
& \leq \frac{1}{2}\left\Vert \partial _{t}\varphi \left( t\right)
\right\Vert _{H}^{2}+C_{\delta }^{\prime }\left\Vert \overrightarrow{%
\mathcal{V}}\right\Vert _{V}^{2},  \notag
\end{align}%
for all $t\geq T_{0},$ where%
\begin{equation*}
\overrightarrow{\mathcal{V}}\left( \cdot \right) :=\nabla a\left( \cdot
\right) -2\nabla J\ast \varphi .
\end{equation*}%
Recalling that $J\in W^{2,1}(\mathbb{R}^{d}) $ the second term on the
right-hand side of (\ref{4.8}) is also uniformly (in time) bounded by some
positive constant $C_{\delta ,J}$ (i.e., $\nabla \overrightarrow{\mathcal{V}}%
\in L^{2}\left( \mathbb{R}^{d}\times \mathbb{R}^{d}\right) $). Therefore, we
may integrate (\ref{4.8}) over $\left( t,t+1\right) $ and exploit (\ref{4.6}%
) to deduce (\ref{regh1}) from an application of the uniform Gronwall
inequality.
\end{proof}

\begin{remark}
Exploiting a suitable version of the {\L }ojasiewicz-Simon inequality (see,
e.g., Lemma \ref{LS_lemma2}), it was proven in \cite[Theorem 2.2]{LP} that
every weak solution $\varphi $ of (\ref{degm})-(\ref{degmic}) converges in
the $H$-metric as time goes to infinity to a single equilibrium%
\begin{equation*}
\begin{array}{ll}
\varphi _{\ast }=1/\left( e^{w_{\ast }-\mu _{\ast }}+1\right) , &
\left\langle \varphi _{\ast }\right\rangle =\left\langle \varphi
_{0}\right\rangle , \\
w_{\ast }\left( x\right) =a\left( x\right) -2J\ast \varphi _{\ast }, & \mu
_{\ast }=\text{constant.}%
\end{array}%
\end{equation*}%
In view of (\ref{regh1}), this can be improved to a convergence rate in the $%
L^{s}\cap H^{1-\nu }$-metric for any $s\geq 2$ and $\nu \in \left(
0,1\right) $, i.e.,
\begin{equation*}
\left\Vert \varphi \left( t\right) -\varphi _{\ast }\right\Vert _{L^{p}\cap
H^{1-\nu }}\sim \left( 1+t\right) ^{-\frac{1}{\rho }},\text{ as }%
t\rightarrow \infty ,
\end{equation*}%
thanks to (\ref{sep_deg}), and the additional assumption (\ref{addit_assum})
(cf. also Remark \ref{rem_regp}).
\end{remark}

\begin{remark}
The assumption on $J$ in (\ref{addit_assum}) can be actually relaxed to also
include Newtonian and Bessel-like potentials. In general, the second
derivatives of such potentials are not locally integrable, but one may still
properly define $\nabla ^{2}J\ast \varphi $ as a bounded distribution on $%
L^{p}\left( \Omega \right) ,$ $1<p<\infty $, using the Calder\'{o}n--Zygmund
theory. In particular, for such distributions there holds $\nabla
\overrightarrow{\mathcal{V}}\in L^{p}$ for every $1<p<\infty $ (see, e.g.,
\cite[Lemma 2]{BRB}, \cite[Lemma 2.1]{G1}), and thus we can also conclude (%
\ref{regh1}) for such potentials.
\end{remark}

\begin{remark}
It was proved in \cite{LP2} that solutions of (\ref{degm})-(\ref{degmic})
also satisfy%
\begin{equation*}
\varphi \left( t\right) \in L^{\infty }\left( [T_{0},\infty );W^{2,2}\left(
\Omega \right) \right) \cap L^{\infty }\left( [T_{0},\infty );C^{\beta
}\left( \overline{\Omega }\right) \right) ,\text{ }\beta <\frac{1}{2},
\end{equation*}%
after the separation time $T_{0}>0$ provided that $\left\Vert \partial
_{t}\varphi \left( T_{0}\right) \right\Vert _{\left( H^{1}\right) ^{\ast }}$
is finite and, in addition,%
\begin{equation*}
\left\Vert J\ast u\right\Vert _{W^{2,2}\left( \Omega \right) }\leq
C_{J}\left\Vert u\right\Vert _{W^{1,2}\left( \Omega \right) }.
\end{equation*}%
This is a \emph{conditional} result which requires \emph{stronger}
assumptions on the kernel $J$ and on $\Omega $. For instance, to prove the
condition $\left\Vert \partial _{t}\varphi \left( T_{0}\right) \right\Vert
_{\left( H^{1}\right) ^{\ast }}<\infty $\ one needs to show that $\left\Vert
\varphi \left( T_{0}\right) \right\Vert _{V}<\infty .$
\end{remark}

\bigskip \noindent \textbf{Acknowledgments.} The authors thank the anonymous
referees for their careful reading of the manuscript.

\end{document}